\documentclass[a4paper,UKenglish,cleveref, autoref,mathscr]{lipics-v2019}
\usepackage{bbm, tikz, mathtools, thm-restate}
\usetikzlibrary{arrows,calc,automata,intersections}
\tikzset{LMC style/.style={>=angle 60,every edge/.append style={thick},every state/.style={thick,minimum size=20,inner sep=0.5}}}

\usepackage{asymptote}
\usepackage{pgfplots, pgfplotstable}

\newcommand{\RR}{\mathbb{R}}

\newcommand{\EE}{\mathbb{E}}
\newcommand{\NN}{\mathbb{N}}

\newcommand{\QQ}{\mathbb{Q}}
\newcommand{\PP}{\mathbb{P}}

\newcommand{\e}{\mathrm{e}}

\newcommand{\GG}{\mathscr{G}}

\newcommand{\1}{\mathbbm{1}}

\newcommand{\supp}{\mathrm{supp}}

\newcommand{\liexp}{\lim_{n\rightarrow\infty} \frac1n \ln L_n}

\newcommand{\Psimin}{\Psi_{\text{min}}}
\newcommand{\SPRT}{\mathrm{SPRT}}

\newcommand{\A}{\mathcal{A}}
\newcommand{\B}{\mathcal{B}}
\renewcommand{\epsilon}{\varepsilon}
\renewcommand{\H}{\mathcal{H}}
\newcommand{\M}{\mathcal{M}}
\newcommand{\oA}{\bar{\A}}
\renewcommand{\S}{\mathcal{S}}
\newcommand{\uparr}[1]{#1\mathord{\uparrow}}

\newcommand*{\DrawBoundingBox}[1][]{%
    \draw [
        draw=gray,
    ]
    ([shift={(-1pt,-1pt)}]current bounding box.south west)
    rectangle ([shift={(1pt,1pt)}]current bounding box.north east);
}

\graphicspath{ {images/} }

\DeclarePairedDelimiter\floor{\lfloor}{\rfloor}

\newcommand{\stefan}[1]{\marginpar{\textcolor{blue}{#1}}}
\renewcommand{\stefan}[1]{}

\title{On the Sequential Probability Ratio Test in Hidden Markov Models}

\author{Oscar Darwin}{Department of Computer Science, Oxford University, United Kingdom }{}{https://orcid.org/0000-0001-5016-014X}{}%TODO mandatory, please use full name; only 1 author per \author macro; first two parameters are mandatory, other parameters can be empty. Please provide at least the name of the affiliation and the country. The full address is optional

\author{Stefan Kiefer}{Department of Computer Science, Oxford University, United Kingdom}{}{https://orcid.org/0000-0003-4173-6877}{}

\authorrunning{O. Darwin and S. Kiefer}%TODO mandatory. First: Use abbreviated first/middle names. Second (only in severe cases): Use first author plus 'et al.'

\Copyright{John Q. Public and Joan R. Public}%TODO mandatory, please use full first names. LIPIcs license is "CC-BY";  http://creativecommons.org/licenses/by/3.0/


%TODO mandatory; please add comma-separated list of keywords

\category{}%optional, e.g. invited paper

\relatedversion{}%optional, e.g. full version hosted on arXiv, HAL, or other respository/website
\supplement{}%optional, e.g. related research data, source code, ... hosted on a repository like zenodo, figshare, GitHub, ...

\nolinenumbers %uncomment to disable line numbering

\EventEditors{John Q. Open and Joan R. Access}
\EventNoEds{2}
\EventLongTitle{42nd Conference on Very Important Topics (CVIT 2016)}
\EventShortTitle{CVIT 2016}
\EventAcronym{CVIT}
\EventYear{2016}
\EventDate{December 24--27, 2016}
\EventLocation{Little Whinging, United Kingdom}
\EventLogo{}
\SeriesVolume{42}
\ArticleNo{23}
\begin{document}

\maketitle

\begin{abstract}
We consider the Sequential Probability Ratio Test applied to Hidden Markov Models. Given two Hidden Markov Models and a sequence of observations generated by one of them, the Sequential Probability Ratio Test attempts to decide which model produced the sequence. We show relationships between the execution time of such an algorithm and Lyapunov exponents of random matrix systems. Further, we give complexity results about the execution time taken by the Sequential Probability Ratio Test.
\end{abstract}

\section{Introduction}

A (discrete-time, finite-state) \emph{Hidden Markov Model (HMM)} (often called \emph{labelled Markov chain}) has a finite set $Q$ of states and for each state a probability distribution over its possible successor states.
Every state is associated with a probability transition over a successor state and an emitted letter (\emph{observation}).
For example, consider the following HMM:
\begin{center}
\begin{tikzpicture}[scale=2.5,LMC style]
\node[state] (q1) at (0,0) {$q_1$};
\node[state] (q2) at (1,0) {$q_2$};
\path[->] (q1) edge [loop,out=200,in=160,looseness=10] node[left] {$\frac13 a$} (q1);
\path[->] (q1) edge [bend left] node[above] {$\frac23 b$} (q2);
\path[->] (q2) edge [loop,out=20,in=-20,looseness=10] node[right] {$\frac23 a$} (q2);
\path[->] (q2) edge [bend left] node[below] {$\frac13 b$} (q1);
\end{tikzpicture}
\end{center}
In state~$q_1$, the probability of emitting~$a$ and the next state being also~$q_1$ is~$\frac13$, and the probability of emitting~$b$ and the next state being~$q_2$ is~$\frac23$.
An HMM is typically viewed as a producer of a finite or infinite word of emitted observations.
For example, starting in~$q_1$, the probability of producing a word with prefix $a b a$ is $\frac13 \cdot \frac23 \cdot \frac23$, whereas starting in~$q_2$, the probability of $a b a$ is $\frac23 \cdot \frac13 \cdot \frac13$.
The random sequence of states is considered not observable (which explains the term \emph{hidden} in HMM).

HMMs are widely employed in fields such as speech recognition (see~\cite{Rabiner89} for a tutorial),
gesture recognition~\cite{Gesture},
%musical score following~\cite{MusicalScore},
signal processing~\cite{SignalProcessing},
and climate modeling~\cite{Weather}.
HMMs are heavily used in computational biology~\cite{HMM-comp-biology},
more specifically in DNA modeling~\cite{DNA-modeling} and biological sequence analysis~\cite{durbin1998biological},
including protein structure prediction~\cite{ProteinStructure} %, detecting similarities in genomes~\cite{Homology}
and gene finding~\cite{GeneFinding}.
In computer-aided verification, HMMs are the most fundamental model for probabilistic systems; model-checking tools such as Prism~\cite{KNP11} or Storm~\cite{Storm} are based on analyzing HMMs efficiently.

One of the most fundamental questions about HMMs is whether two initial distributions are \emph{(trace) equivalent}, i.e., generate the same distribution on infinite observation sequences.
In the example above, we argued that (the Dirac distributions on) the states $q_1, q_2$ are not equivalent.
The equivalence problem is very well studied and can be solved in polynomial time using algorithms that are based on linear algebra~\cite{schut61,Paz71,Tzeng92,CortesMRdistance}.
The equivalence problem has applications in verification, e.g., of randomised anonymity protocols~\cite{kief11}.

Equivalence is a strong notion, and a natural question about nonequivalent distributions in a given HMM is \emph{how} different they are.
For initial distributions $\pi_1, \pi_2$ on the states of the HMM, let us write $\PP_{\pi_1}, \PP_{\pi_2}$ for the induced probability measure on infinite observation sequences; i.e., $\PP_{\pi_i}(E)$, for a measurable event $E \subseteq \Sigma^\omega$, is the probability that the random infinite word~$w \in \Sigma^\omega$ produced starting from~$\pi_i$ is in~$E$.
Then, the \emph{total variation distance} between $\PP_{\pi_1}, \PP_{\pi_2}$ is defined as
\[
 d(\pi_1,\pi_2) \ := \ \sup~\{ |\PP_{\pi_1}(E) - \PP_{\pi_2}(E)| \mid \text{measurable } E \subseteq \Sigma^\omega\}\,.
\]
This supremum is a maximum; i.e., there always exists a ``maximizing event'' $E \subseteq \Sigma^\omega$ with $d(\pi_1,\pi_2) = \PP_{\pi_1}(E) - \PP_{\pi_2}(E)$.
In these terms, initial distributions $\pi_1, \pi_2$ are equivalent if and only if $d(\pi_1, \pi_2) = 0$.
The total variation distance was studied in more detail in~\cite{kief14}.
There it was shown that the problem whether $d(\pi_1, \pi_2) = 1$ holds can also be decided in polynomial time.
Call distributions $\pi_1, \pi_2$ \emph{distinguishable} if $d(\pi_1, \pi_2) = 1$.
Distinguishability was used for runtime monitoring~\cite{kief16} and diagnosability~\cite{BertrandHL16,AkshayBFG19} of stochastic systems.

%Although it is easy to come up with HMMs and distributions $\pi_1,\pi_2$ such that $0 < d(\pi_1,\pi_2) < 1$,
Distributions $\pi_1, \pi_2$ that are distinguishable (i.e., $d(\pi_1,\pi_2) = 1$) can nevertheless be ``hard'' to distinguish.
In our example above, (the Dirac distributions on) $q_1, q_2$ are distinguishable.
If we replace the transition probabilities $\frac13, \frac23$ in the HMM by $\frac12-\varepsilon, \frac12+\varepsilon$, respectively, states $q_1, q_2$ remain distinguishable for every $\varepsilon>0$, although, intuitively, the smaller $\varepsilon > 0$ the more observations are needed to define an event~$E$ such that $\PP_{\pi_1}(E) - \PP_{\pi_2}(E)$ is close to~$1$.

To make this more precise, for initial distributions $\pi_1, \pi_2$, a word $w \in \Sigma^\omega$ and $n \in \NN$ consider the \emph{likelihood ratio}
\[
L_n(w) \ := \ \frac{\PP_{\pi_1}(w_n \Sigma^\omega)}{\PP_{\pi_2}(w_n \Sigma^\omega)}\,,
\]
where $w_n$ denotes the length-$n$ prefix of~$w$.
In the example above, we argued that $\PP_{q_1}(a b a \Sigma^\omega) = \frac13 \cdot \frac23 \cdot \frac23$ and $\PP_{q_2}(a b a \Sigma^\omega) = \frac23 \cdot \frac13 \cdot \frac13$.
Thus, for any word $w$ starting with $a b a$ we have $L_n(w) = 2$. We consider the likelihood ratio~$L_n$ as a random variable for every $n \in \NN$. 
%Since the word~$w$ is produced randomly, the likelihood ratio~$L_n$ is a random variable for every $n \in \NN$. 
It turns out more natural to focus on the \emph{log-likelihood ratio} $\ln L_n$.
One can show that the limit $\lim_{n \to \infty} \ln L_n \in [-\infty, \infty]$ exists $\PP_{\pi_1}$-almost surely and $\PP_{\pi_2}$-almost surely (see, e.g., \cite[Proposition~6]{kief14}).
%Moreover, the event $\{\overline{L} \ge 0\}$ is a ``maximizing event'' referred to above, i.e., $d(\pi_1,\pi_2) = \PP_{\pi_1}(\{\overline{L} \ge 0\}) - \PP_{\pi_2}(\{\overline{L} \ge 0\})$ (see, e.g., \cite[Theorem~11]{kief14}).
%In the case of distinguishable $\pi_1, \pi_2$ it follows that $\PP_{\pi_1}(\{\overline{L} \ge 0\}) = 1$ and $\PP_{\pi_2}(\{\overline{L} \ge 0\}) = 0$;
In fact, if $\pi_1, \pi_2$ are distinguishable, then $\lim_{n \to \infty} \ln L_n = \infty$ holds $\PP_{\pi_1}$-almost surely and $\lim_{n \to \infty} \ln L_n = -\infty$ holds $\PP_{\pi_2}$-almost surely.
This suggests the ``average slope'', $\liexp$, of increase or decrease of~$\ln L_n$ as a measure of \emph{how} distinguishable two distinguishable distributions $\pi_1,\pi$ are.
%Let us formalise \emph{average slope} as $\liexp$.

The log-likelihood ratio plays a central role in the \emph{sequential probability ratio test (SPRT)}~\cite{wald45}, which is  optimal~\cite{WaldWolfowitz48} among sequential hypothesis tests (such tests attempt to decide between two hypotheses without fixing the sample size in advance).
In terms of an HMM and two initial distributions $\pi_1, \pi_2$, the SPRT attempts to decide, given longer and longer prefixes of an observation sequence $w \in \Sigma^\omega$, which of $\pi_1, \pi_2$ is more likely to emit~$w$.
The SPRT works as follows: fix a lower and an upper threshold (which determine type-I and type-II errors); given increasing prefixes of~$w$ keep track of $\ln L_n(w)$, and when the upper threshold is crossed output~$\pi_1$ and stop, and when the lower threshold is crossed output~$\pi_2$ and stop.
Again, it is natural to assume that the average slope of increase or decrease of~$\ln L_n$ determines how long the SPRT needs to cross one of the thresholds.

If the average slope $\liexp$ exists and equals a number $\ell$ with positive probability, we call $\ell$ a \emph{likelihood exponent}.
The term is motivated by a close relationship to \emph{Lyapunov exponents}, which characterise the growth rate of certain random matrix products.
As the most fundamental contribution of this paper, we show that the average slope exists almost surely and that any HMM with $m$~states has at most $m^2 + 1$ likelihood exponents.

The rest of the paper is organised as follows.
In \cref{liexpsubsect} we exhibit a tight connection between the SPRT and likelihood exponents; i.e., the time taken by the SPRT depends on the likelihood exponents of the HMM.
This connection motivates our results on likelihood exponents in the rest of the paper.
In \cref{sec:qual} we prove complexity results concerning the probability that the average slope equals a particular likelihood exponent.
In \cref{sec:rep} we show that the average slope exists almost surely and prove our bound on the number of likelihood exponents.
Further, we show that the likelihood exponents can be efficiently expressed in terms of Lyapunov exponents.
In \cref{sec:det} we show that for \emph{deterministic} HMMs one can compute likelihood exponents in polynomial time.
We conclude in \cref{sec:conclusions}.

\section{Preliminaries} \label{sec:prelims}
We write $\NN$ for the set of non-negative integers. %, $\QQ$ for the set of rationals and $\QQ_+$ for the set of positive rationals.
For $d \in \NN$ we write %and a finite set $Q$ we use the notation $|Q|$ for the number of elements in $Q$,
$[d] = \{1, \dots, d\}$. % and $[Q] = \{1, \dots, |Q|\}$.
For a finite set~$Q$, vectors $\mu \in \RR^Q$ are viewed as row vectors, and their transpose (a column vector) is denoted by $\mu^\top$.
The norm $\|\mu\|$ is assumed to be the $l_1$ norm: $\| \mu \| = \sum_{q \in Q} | \mu_q |$.
We write $\vec{0}, \vec{1}$ for the vectors all of whose entries are $0$, $1$, respectively.
For $q \in Q$, we denote by~$e_q \in \{0,1\}^Q$ the vector with $(e_q)_q = 1$ and $(e_q)_{q'} = 0$ for $q' \ne q$.
A matrix $M \in [0,1]^{Q \times Q}$ is \emph{stochastic} if $\vec{1}^\top = M \vec{1}^\top$.
We often identify vectors $\mu \in [0,1]^Q$ such that $\| \mu \| = 1$ with the corresponding probability distribution on~$Q$.
For $\mu \in [0,\infty)^Q$ we write $\supp(\mu) := \{q \in Q \mid \mu_q > 0\}$.

For a finite alphabet~$\Sigma$ and $n \in \NN$ we denote by $\Sigma^n, \Sigma^*, \Sigma^+, \Sigma^\omega$ the sets of length-$n$ words, finite words, non-empty finite words, infinite words, respectively.
For $w \in \Sigma^\omega$ we write $w_n$ for the length-$n$ prefix of~$w$.

A \emph{Hidden Markov Model} (HMM) is a triple $\H = (Q, \Sigma, \Psi)$ where $Q$ is a finite set of states, $\Sigma$ is a set of observations (or ``letters''), and the function $\Psi : \Sigma \rightarrow [0,1]^{Q \times Q}$ specifies the transitions such that $\sum_{a \in \Sigma} \Psi(a)$ is stochastic. For computational purposes we assume the numbers in $\Psi$ are rational and expressed as fractions of integers encoded in binary.
A \emph{Markov chain} is a pair $(Q, T)$ where $Q$ is a finite set of states and $T \in [0,1]^{Q \times Q}$ is a stochastic matrix.
A Markov chain $(Q,T)$ is naturally associated with its directed \emph{graph} $(Q,\{(q,r) \mid T_{q,r} > 0\})$, and so we may use graph concepts, such as strongly connected components (SCCs), in the context of a Markov chain. Trivial SCCs are considered SCCs.
The \emph{embedded} Markov chain of an HMM $(Q, \Sigma, \Psi)$ is the Markov chain $(Q, \sum_{a \in \Sigma} \Psi(a))$.
We say that an HMM is \emph{strongly connected} if the graph of its embedded Markov chain is.
\begin{example} \label{ex-HMMdef}
The HMM from the introduction is the triple $\H = (\{q_1, q_2\}, \{a,b\}, \Psi)$ with $\Psi(a) = \begin{pmatrix} \frac13 & 0 \\ 0 & \frac23 \end{pmatrix}$ and $\Psi(b) = \begin{pmatrix} 0 & \frac23 \\ \frac13 & 0 \end{pmatrix}$.
The embedded Markov chain is $(\{q_1,q_2\}, \begin{pmatrix} \frac13 & \frac23 \\ \frac13 & \frac23 \end{pmatrix})$.
\end{example}

Fix an HMM $\H = (Q, \Sigma, \Psi)$ for the rest of the section.
We extend $\Psi$ to the mapping $\Psi : \Sigma^* \rightarrow [0,1]^{Q \times Q}$ with $\Psi(a_1 \cdots a_n) = \Psi(a_1) \cdot \ldots \cdot \Psi(a_n)$ and $\Psi(\epsilon) = I$, where $\epsilon$ is the empty word and $I$ the $Q \times Q$ identity matrix.
We call a finite sequence $v = q_0 a_1 q_1 \cdots a_n q_n \in Q (\Sigma Q)^*$ a \emph{path} and $v (\Sigma Q)^\omega$ a \emph{cylinder set} and an infinite sequence $q_0 a_1 q_1 a_2 q_2 \cdots \in Q (\Sigma Q)^\omega$ a \emph{run}.
To~$\H$ and an \emph{initial probability distribution} $\pi \in [0,1]^Q$ we associate the probability space $(Q(\Sigma Q)^\omega, \GG^*, \PP_\pi)$ where $\GG^*$ is the $\sigma$-algebra generated by the cylinder sets and $\PP_\pi$ is the unique probability measure with $\PP_\pi(q_0 a_1 q_1 \cdots a_n q_n (\Sigma Q)^\omega) = \pi_{q_0} \prod_{i=1}^{n} \Psi(a_i)_{q_{i-1},q_i}$.
As the states are often irrelevant, for $E \subseteq \Sigma^\omega$ and $\uparr{E} := \{q_0 a_1 q_1 a_2 q_2 \cdots \mid a_1 a_2 \cdots \in E\} \in \GG^*$ we view also $E$ as an event and may write $\PP_{\pi}(E)$ to mean $\PP_{\pi}(\uparr{E})$.
In particular, for $w \in \Sigma^*$ we have $\PP_{\pi}(w \Sigma^\omega) = \| \pi \Psi(w) \|$.
For $E \subseteq \Sigma^\omega$ we write $\1_E$ for the indicator random variable with $\1_E(w) = 1$ if $w \in E$ and $\1_E(w) = 0$ if $w \not\in E$.
By~$\EE_\pi$ we denote the expectation with respect to~$\PP_{\pi}$. If $\pi$ is the Dirac distribution on state $q$, then we write $\EE_{q}$.

%We say that $A \subseteq \Sigma^n$ is a \emph{cylinder set} if $A = A_1 \times \dots \times A_n$ and $A_i \in \GG$ for $i \in [n]$. For every $n$ there is an induced measure space $(\Sigma^n, \GG^n, \lambda^n)$ where $\GG^n$ is the smallest $\sigma$-algebra containing all cylinder sets in~$\Sigma^n$ and $\lambda^n(A_1 \times \dots \times A_n) = \prod_{i = 1}^n \lambda(A_i)$ for any cylinder set $A_1 \times \dots \times A_n$. Let $A \subset \Sigma^n$ and write $A \Sigma^\omega$ for the set of infinite words over~$\Sigma$ where the first $n$ observations fall in the set~$A$. Given a HMM $(Q, \Sigma, \Psi)$ and initial distribution $\pi$ on $Q$ viewed as vector $\pi \in \RR^{|Q|}$, there is an induced probability space $(\Sigma^\omega, \GG^*, \PP_\pi)$ where $\Sigma^\omega$ is the set of infinite words over~$\Sigma$, and $\GG^*$ is the smallest $\sigma$-algebra containing (for all $n \in \NN$) all sets $A \Sigma^\omega$ where $A\subseteq \Sigma^n$ is a cylinder set and $\PP_\pi$ is the unique probability measure such that
%$\PP_\pi(A \Sigma^\omega) =  \sum_{w \in A} \| \pi  \Psi(w) \|$ for any cylinder set $A \subseteq \Sigma^n$.

A Markov chain $(Q,T)$ and an initial distribution $\iota \in [0,1]^Q$ are associated with a probability measure $\PP_{\iota}$ on measurable subsets of $Q^\omega$; the construction of the probability space is similar to HMMs, without the observation alphabet~$\Sigma$.

Let $(Q, \Sigma, \Psi)$ be an HMM and let $\pi_1, \pi_2$ be two initial distributions. The \emph{total variation distance} is
$d(\pi_1, \pi_2) := \sup_{\uparr{E} \in \GG^*} | \PP_{\pi_1}(E) - \PP_{\pi_2}(E) |$.
This supremum is actually a maximum due to Hahn's decomposition theorem; i.e., there is an event $E \subseteq \Sigma^\omega$ such that $d(\pi_1, \pi_2) = \PP_{\pi_1}(E) - \PP_{\pi_2}(E)$.
We call $\pi_1$ and $\pi_2$ \emph{distinguishable} if $d(\pi_1, \pi_2) = 1$.
Distinguishability is decidable in polynomial time~\cite{kief14}.

%One could define the distinguishability of two pairs $(C_1,\pi_1)$ and $(C_2,\pi_2)$ where $C_i = (Q_i, \Sigma, \Psi_i)$ are HMMs and $\pi_i$ are initial distributions for $i=1,2$.
%We do not need that though, as we can define, in a natural way, a single HMM over the disjoint union of $Q_1$ and~$Q_2$ and consider instead the distinguishability of $\pi_1$ and~$\pi_2$ (where $\pi_1,\pi_2$ are appropriately padded with zeros).

%-- For the rest of the paper we assume that $(Q, \Sigma, \Psi)$ is an HMM. $L_n$ depending on $\pi_i$ etc.

Let $\pi_1$ and $\pi_2$ be initial distributions.
For $n \in \NN$, the \emph{likelihood ratio} $L_n$ is a random variable on $\Sigma^\omega$ given by $L_n(w) = \frac{\| \pi_1 \Psi(w_n) \|}{\| \pi_2 \Psi(w_n) \|}$.
%By \cite[Proposition 6]{kief14} we have that $L_0, L_1, \ldots$ is a martingale and the following lemma holds due to Doob's forward convergence theorem\cite{will91}.
Based on results from~\cite{kief14} we have the following lemma.
%Using another result from~\cite{kief14} one can show the following lemma characterizing when the limit is~$0$.
\begin{restatable}{lemma}{convergenceLn}\label{convergenceLn}
Let $\pi_1, \pi_2$ be initial distributions.
\begin{enumerate}
\item $\lim_{n \rightarrow \infty} L_n$ exists $\PP_{\pi_2}$-almost surely and lies in $[0,\infty)$.
\item $\lim_{n \rightarrow \infty} L_n = 0 \ \ \PP_{\pi_2}$-almost surely if and only if $\pi_1$ and $\pi_2$ are distinguishable.
\end{enumerate}
\end{restatable}

\begin{example}\label{sleepcycles}
We illustrate convergence of the likelihood ratio using an example from \cite{rockhart13} where the authors use HMMs to model sleep cycles.
They took measurements of 51 healthy and 51 diseased individuals and using electrodes attached to the scalp, they read electrical signal data as part of an electroencephalography (EEG) during sleep.
They split the signal into 30 second intervals and mapped each interval onto the simplex $\Delta^3 = \{(x_1, x_2, x_3, x_4) \in [0,1]^4 \mid \sum_{i = 1}^4 x_i = 1\}$.
For each individual this results in a time series of points in $\Delta^3$.
They modelled this data using two HMMs, each with 5 states, for healthy and diseased individuals using a numerical maximum likelihood estimate.
Each state is associated with a probability density function describing the distribution of observations in~$\Delta^3$.
We describe in \cref{app:sleepcycles} how we obtained from this an HMM $\H = (Q, \Sigma, \Psi)$ with (finite) observation alphabet $\Sigma = \{a_1, \ldots, a_5\}$ and two initial distributions $\pi_1, \pi_2$ corresponding to healthy and diseased individuals, respectively.
Using the algorithm from \cite{kief14} one can show that $\pi_1$ and $\pi_2$ are distinguishable.

We sampled runs of~$\H$ started from $\pi_1$ and $\pi_2$ and plotted the corresponding sequences of $\ln L_n$.
We refer to each of these two plots as a \emph{log-likelihood plot}; see \Cref{loglikes}.
\begin{center}
	\begin{figure}[ht]
		\includegraphics[width=\textwidth/2]{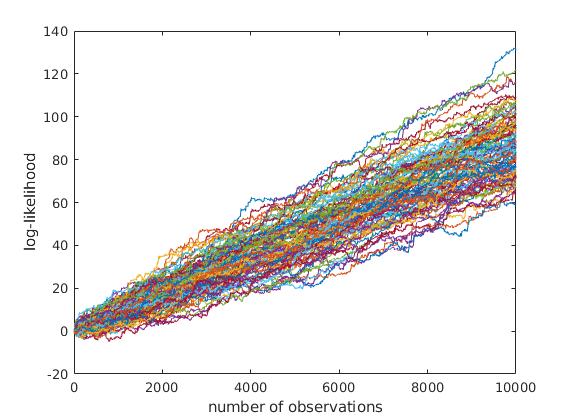}	
		\includegraphics[width=\textwidth/2]{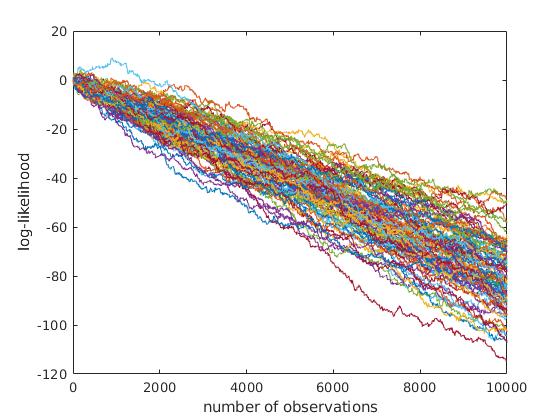}
		\caption{The two images show two log-likelihood plots of sample runs produced by $\pi_1$ and $\pi_2$, respectively.}\label{loglikes}
	\end{figure}
\end{center}
By \Cref{convergenceLn}.2 it follows that $\ln L_n$ converges $\PP_{\pi_1}$-a.s.\ (almost-surely) to $\infty$ and $\PP_{\pi_2}$-a.s.\ to~$-\infty$.
This is affirmed by \Cref{loglikes}.
Both log-likelihood plots also appear to follow a particular slope.
This suggests that we can distinguish between words produced by $\pi_1$ and~$\pi_2$ by tracking the value of $\ln L_n$ to see whether it crosses a lower or upper threshold.
This is the intuition behind the \emph{Sequential Probability Ratio Test} (SPRT).
%The amount of steps this process takes should also be related to the slope of the slopes.
\end{example}

\section{Sequential Probability Ratio Test}\label{liexpsubsect}

Fix an HMM $H = (Q,\Sigma,\Psi)$ for the rest of the paper.
Given initial distributions $\pi_1, \pi_2$ and error bounds $\alpha, \beta \in (0,1)$, the SPRT runs as follows.
%We specify two error probabilities $\alpha, \beta \in (0,1)$ such that the SPRT gives the (wrong) result ``$\pi_2$'' from observations produced by~$\pi_1$ with probability at most~$\alpha$, and the (wrong) result ``$\pi_1$'' from observations produced by~$\pi_2$ with probability at most~$\beta$.
It continues to read observations and computes the value of $\ln L_n$ until $\ln L_n$ leaves the interval $[A,B]$, where $A := \ln \frac{\alpha}{1 - \beta}$ and $B := \ln \frac{1 - \alpha}{\beta}$.
If $\ln L_n \leq A$ the test outputs ``$\pi_2$'' and if $\ln L_n \geq B$ the test outputs ``$\pi_1$''.
We may view the SPRT as a random variable $\SPRT_{\alpha, \beta} : \Sigma^\omega \rightarrow \{\pi_1, \pi_2, ?\}$, where $?$ denotes that the SPRT does not terminate, i.e., $\ln L_n \in [A,B]$ for all~$n$.
We have the following correctness property.
\begin{restatable}{proposition}{sprtcorrectness}\label{sprtcorrectness}
Suppose $\pi_1$ and $\pi_2$ are distinguishable.
Let $\alpha, \beta \in (0,1)$.
By choosing $A = \ln \frac{\alpha}{1 - \beta}$ and $B = \ln \frac{1 - \alpha}{\beta}$, we have $\PP_{\pi_1} (\SPRT_{\alpha, \beta} = \pi_2) \leq \alpha$ and $\PP_{\pi_2} (\SPRT_{\alpha, \beta} = \pi_1) \leq \beta$.
\end{restatable}
In the following we consider the SPRT with respect to the measure $\PP_{\pi_2}$.
This is without loss of generality as there is a dual version of the $\SPRT$, say $\overline{\SPRT}$ with $\overline{L}_n = 1/L_n$ instead of $L_n$, such that $ \overline{\SPRT}_{\beta, \alpha} = \SPRT_{\alpha, \beta}$.
Define the stopping time
\begin{equation*}
N_{\alpha, \beta} \ := \ \min \{n \in \NN \mid \ln L_n \not\in [A, B]\} \ \in \ \NN \cup \{\infty\}\,.
\end{equation*}
We have that $N_{\alpha, \beta}$ is monotone decreasing in the sense that for $\alpha \leq \alpha'$ and $\beta \leq \beta'$ we have $N_{\alpha, \beta} \geq N_{\alpha', \beta'}$.
When $\pi_1$ and $\pi_2$ are distinguishable, $N_{\alpha, \beta}$ is $\PP_{\pi_2}$-a.s.\ finite by \Cref{convergenceLn}.2.

\subsection{Expectation of $N_{\alpha, \beta}$}
Consider the two-state HMM where $p_1 \neq p_2$.
\begin{center}
	\begin{tikzpicture}[scale=2.3,LMC style]
	\node[state] (s0) at (-1.25,0) {$s_1$};
	\node[state] (s1) at (1.25,0) {$s_2$};
	
	\path[->] (s0) edge [loop,out=200,in=160,looseness=10] node[pos=0.5,left] {$p_1 : a$} (s0);
	\path[->] (s0) edge [loop,out=20,in=340,looseness=10] node[pos=0.5,right] {$(1 - p_1) : b$} (s0);
	\path[->] (s1) edge [loop,out=200,in=160,looseness=10] node[pos=0.5,left] {$p_2 : a$} (s1);
	\path[->] (s1) edge [loop,out=20,in=340,looseness=10] node[pos=0.5,right] {$(1 - p_2) : b$} (s1);
	\end{tikzpicture}
\end{center}
(The Dirac distributions of) $s_1$ and $s_2$ are distinguishable. Further, the increments $\ln L_{n + 1} - \ln L_n$ are independent and identically distributed (i.i.d.) and $0 > \EE_{s_2}[\ln L_{n+1} - \ln L_n] = p_2 \ln \frac{p_1}{p_2} + (1 - p_2)\ln \frac{1 - p_1}{1 - p_2} =: \ell$. Intuitively as $\ell$ gets more negative, the HMMs become more different.\footnote{In fact, $\ell$ is the \emph{KL-divergence} of the distributions $f_1, f_2$ where $f_i(a) = p_i$ and $f_i(b) = 1 - p_i$ for $i = 1, 2$.} Indeed, Wald~\cite{wald45} shows that the expected stopping time $\EE_{s_2}[N_{\alpha, \beta}]$ and $\ell$ are inversely proportional:
\begin{equation}\label{singletonexptime}
\EE_{s_2}[N_{\alpha, \beta}] = \frac{\beta \ln \frac{1 - \alpha}{\beta} + (1 - \beta) \ln \frac{\alpha}{1 - \beta}}{\ell}.
\end{equation}
This Wald formula cannot hold in general for (multi-state) HMMs. The increments $\ln L_{n+1} - \ln L_n$ need not be independent and $\EE_{s_2}[\ln L_{n+1} - \ln L_n]$ can be different for different $n$. Further, $| \ln L_{n+1} - \ln L_n |$ can be unbounded; cf.~\cite[Example~6]{kief16}.

Nevertheless, in \Cref{loglikes} we observed that $\ln L_n$ appears to decrease linearly (on the $\pi_2$ plot).
Indeed, we show in \Cref{liexplimits} below that the limit $\liexp$ exists $\PP_{\pi_2}$-almost surely.
Intuitively it corresponds to the average slope of the log-likelihood plot for $\pi_2$.
In the two-state case, there is a simple proof of this using the law of large numbers:
\begin{equation*}
\liexp = \lim_{n \rightarrow \infty}\frac1n \sum_{i = 0}^{n - 1} [\ln L_{i + 1} - \ln L_i ] = \EE_{\pi_2} [\ln L_1 - \ln L_0] = \ell \ \ \PP_{\pi_2}\text{-a.s.}
\end{equation*}
The number $\ell$ is called a likelihood exponent, as defined generally in the following definition.
\begin{definition}
For initial distributions $\pi_1, \pi_2$, a number $\ell \in [-\infty, 0]$ is a \emph{likelihood exponent} if $\PP_{\pi_2}(\liexp = \ell) > 0$.
\end{definition}
By \cref{convergenceLn}.1 we have $\PP_{\pi_2}(\liexp > 0) = 0$, as $\PP_{\pi_2}(\lim_{n \rightarrow \infty} L_n < \infty) = 1$. Hence, we may restrict likelihood exponents to~$[-\infty, 0]$.
%We also have:
%\begin{restatable}{lemma}{lemneginf}\label{lem:neginf}
%
%\end{restatable}
We write $\Lambda_{\pi_1, \pi_2} \subseteq [-\infty,0]$ for the set of likelihood exponents for $\pi_1, \pi_2$ and define $\Lambda := \bigcup_{\pi_1, \pi_2} \Lambda_{\pi_1, \pi_2}$; i.e., $\Lambda$ depends only on the HMM~$\H$.
For $\ell \in \Lambda$ we define the event $E_\ell = \{\liexp = \ell\}$.

\begin{example}\label{estimatingsleepcyclelikelihood}
In the case of \Cref{sleepcycles} we have $\Lambda_{\pi_1, \pi_2} = \{\ell\}$ where the slope of the right hand side of \Cref{loglikes} suggests that $\ell \approx -\frac{80}{10000} = -0.008$.
\end{example}

\begin{example}\label{multilimitliexp}
Even for fixed $\pi_1, \pi_2$ there may be multiple likelihood exponents.
Consider the following HMM with initial Dirac distributions $\pi_1 = e_{s_1}$ and $\pi_2 = e_{s_4}$.
\begin{center}
	\begin{tikzpicture}[scale=2.3,LMC style]
\useasboundingbox (-2.5,-0.35) rectangle (2,0.35);
	\node[state] (s0) at (-1.25,0) {$s_1$};
	\node[state] (s1) at (-2,0) {$s_2$};
	\node[state] (s2) at (-0.5,0) {$s_3$};
	\node[state] (s3) at (1.5,0) {$s_4$};
	
	\path[->] (s0) edge node[pos=0.5,above] {$\frac14 a$} (s1);
	\path[->] (s0) edge node[pos=0.5,above] {$\frac34 b$} (s2);
	
	\path[->] (s1) edge [loop,out=220,in=260,looseness=10] node[pos=0.5,left,yshift=5,xshift=-1] {$\frac23 b$} (s1);
	\path[->] (s1) edge [loop,out=100,in=140,looseness=10] node[pos=0.5,left,yshift=-5,xshift=-1] {$\frac13 a$} (s1);
	
	\path[->] (s2) edge [loop,out=-40,in=-80,looseness=10] node[pos=0.5,right,yshift=5,xshift=1] {$\frac12 b$} (s2);
	\path[->] (s2) edge [loop,out=80,in=40,looseness=10] node[pos=0.5,right,yshift=-5,xshift=1] {$\frac12 a$} (s2);
	
	\path[->] (s3) edge [loop,out=200,in=160,looseness=10] node[pos=0.5,left] {$\frac12 a$} (s3);
	\path[->] (s3) edge [loop,out=20,in=340,looseness=10] node[pos=0.5,right] {$\frac12 b$} (s3);
%\DrawBoundingBox
	\end{tikzpicture}
\end{center}
We observe two different likelihood exponents depending on the first letter produced.
%If the first letter is $a$ then $L_1 = \frac12$ and if the first letter is $b$ then $L_1 = \frac32$.
If the first letter is~$a$ then $\ln L_{n + 1} - \ln L_n$ are i.i.d.\ for $n \ge 1$ and $\liexp = \frac12 \ln \frac{1/3}{1/2} +  \frac12 \ln \frac{2/3}{1/2} = \frac12 \ln \frac89 =: \ell$ like the two-state example above.
If the first letter is $b$ then $L_n = \frac32$ for all $n \ge 1$ and $\liexp = 0$.
Thus, $\Lambda_{\pi_1,\pi_2} = \{\ell, 0\}$ and $\PP_{\pi_2}(E_{\ell}) = \PP_{\pi_2}(E_0) = \frac12$.
\end{example}

The following theorem is perhaps the most fundamental contribution of this paper.
\begin{restatable}{theorem}{liexplimits}\label{liexplimits}
For any initial distributions $\pi_1, \pi_2$ the limit $\liexp$ exists $\PP_{\pi_2}$-almost surely.
Furthermore, we have $|\Lambda | \leq |Q|^2+1$.	
%The set of likelihood exponents satisfies $\Lambda \subset [-\infty, 0]$ and $|\Lambda| \leq |Q|^2$. Further $\liexp$ exists (and by definition is in $\Lambda$) $\PP_{\pi_2}$-almost surely for any $\pi_1, \pi_2$.
\end{restatable}
It follows from a stronger theorem, \cref{thm:likelihood-to-lyapunov}, which we prove in \cref{sec:rep}.

Returning to the SPRT, we investigate how $\liexp$ influences the performance of the SPRT for small $\alpha$ and~$\beta$.
Intuitively we expect a steeper slope in the likelihood plot (cf.~\Cref{loglikes}) to lead to faster termination.
In the two-state case, Wald's formula \eqref{singletonexptime} becomes%
\begin{equation}\label{asymptoticwaldeq}
\EE_{s_2}[N_{\alpha, \beta}] = \frac{\beta \ln \frac{1 - \alpha}{\beta} + (1 - \beta) \ln \frac{\alpha}{1 - \beta}}{\ell} \sim \frac{\ln \alpha}{\ell} \ (\text{as } \alpha, \beta \rightarrow 0),
\end{equation}
where we use the notation $\sim$ defined as follows.
\stefan{Please check. Perhaps move.}
For functions $f,g : (0,\infty) \times (0, \infty) \to (0,\infty)$ we write ``$f(x,y) \sim g(x,y)$ (as $x,y \to 0$)'' to denote that for all $\epsilon > 0$ there is $\delta > 0$ such that for all $x,y \in (0,\delta)$ we have $f(x,y)/g(x,y) = [1-\epsilon,1+\epsilon]$.

In \Cref{asymptoticwald} below we generalise \Cref{asymptoticwaldeq} to arbitrary HMMs.
Indeed a very similar asymptotic identity holds.
In the case that $\Lambda = \{\ell\}$ and $\ell \in (-\infty, 0)$ we have $\EE_{s_2}[N_{\alpha, \beta}] \sim \frac{\ln \alpha}{\ell}$ as $\alpha, \beta \rightarrow 0$.
If $|\Lambda| > 1$ then we condition our expectation on $\liexp$.
\begin{restatable}[Generalised Wald Formula]{theorem}{asymptoticwald}\label{asymptoticwald}
Let $\ell$ be a likelihood exponent and let $\pi_1$ and $\pi_2$ be initial distributions.
\begin{enumerate}
\item If $\ell \in (-\infty, 0)$ then
$\displaystyle
\EE_{\pi_2} \big[ N_{\alpha,\beta} \mid E_{\ell}\big] \sim \frac{\ln \alpha}{\ell} \ \ (\text{as } \alpha, \beta \rightarrow 0)
$.
\item If $\ell = 0$ then there exist $\alpha, \beta > 0$ such that
$\displaystyle
\EE_{\pi_2} \big[ N_{\alpha,\beta} \mid E_\ell\big] = \infty
$.
\item If $\ell = -\infty$ then
$\displaystyle
\sup_{\alpha, \beta}~\EE_{\pi_2} \big[ N_{\alpha,\beta} \mid E_\ell \big] < \infty
$.
\end{enumerate}
\end{restatable}

The theorem above pertains to the expectation of $N_{\alpha, \beta}$.
In the next subsection we give additional information about the distribution of $N_{\alpha, \beta}$, further strengthening the connection between $N_{\alpha, \beta}$ and likelihood exponents.

\subsection{Distribution of $N_{\alpha, \beta}$}
\subsubsection{Likelihood Exponent $0$}
\begin{example}
We continue with \Cref{multilimitliexp} to illustrate the second case in \Cref{asymptoticwald}.
By picking $\alpha = \frac14, \beta = \frac14$ the thresholds for the SPRT are $A = \ln \frac13$ and $B = \ln 3$.
If the first letter is $b$, then $\ln L_n = \ln \frac32$ for all $n > 1$, thus never crosses the SPRT bounds and $\liexp = 0$.
Hence with probability $\frac12$ the SPRT fails to terminate  and $N_{\alpha, \beta} = \infty$.
It follows that $\PP_{\pi_2}(E_0) = \frac12$ and $\EE_{\pi_2}[N_{\alpha, \beta} \mid E_0] = \infty$ and, thus, $\EE_{\pi_2}[N_{\alpha, \beta}] = \infty$.
\end{example}
The second part of \Cref{asymptoticwald} says that the expectation of $N_{\alpha, \beta}$ conditioned under $E_0$ is infinite.
The following proposition strengthens this statement.
Conditioning under~$E_0$, the probability that $N_{\alpha, \beta}$ is infinite converges to~$1$ as $\alpha, \beta \rightarrow 0$.
Recall that $N_{\alpha, \beta}$ is monotone decreasing.
It follows that $\{ N_{\alpha', \beta'} = \infty \} \subseteq \{ N_{\alpha, \beta} = \infty \}$ if $\alpha \leq \alpha'$ and $\beta \leq \beta'$.
\begin{restatable}{proposition}{probexpzero}\label{probexp0}
The following two equalities hold up to $\PP_{\pi_2}$-null sets:
\begin{equation*}
E_0 ~=~ \left\{\lim_{n \rightarrow \infty} L_n > 0 \right\} ~=~ \bigcup_{\alpha, \beta > 0} \left\{N_{\alpha, \beta} = \infty\right\}.
\end{equation*}
Thus, $\lim_{\alpha, \beta \rightarrow 0} \PP_{\pi_2}(N_{\alpha, \beta} = \infty) = \PP_{\pi_2}(E_0)$.
\end{restatable}
\begin{corollary}[using \Cref{convergenceLn}.2] \label{cor:probexp0}
Initial distributions $\pi_1$ and $\pi_2$ are distinguishable if and only if $\PP_{\pi_2}(E_0) = 0$ if and only if $\PP_{\pi_2}(N_{\alpha, \beta} < \infty) = 1$ holds for all $\alpha, \beta > 0$.
\end{corollary}

\subsubsection{Likelihood Exponent $-\infty$}

\begin{example}\label{mortalitywaldex}
Consider now a modification of \Cref{multilimitliexp} where state $s_3$ has the $b$ loop removed.
\begin{center}
	\begin{tikzpicture}[scale=2.3,LMC style]
\useasboundingbox (-2.5,-0.35) rectangle (2,0.35);
	\node[state] (s0) at (-1.25,0) {$s_1$};
	\node[state] (s1) at (-2,0) {$s_2$};
	\node[state] (s2) at (-0.5,0) {$s_3$};
	\node[state] (s3) at (1.5,0) {$s_4$};
	
	\path[->] (s0) edge node[pos=0.5,above] {$\frac14 a$} (s1);
	\path[->] (s0) edge node[pos=0.5,above] {$\frac34 b$} (s2);
	
	\path[->] (s1) edge [loop,out=220,in=260,looseness=10] node[pos=0.5,left,yshift=5,xshift=-2] {$\frac23 b$} (s1);
	\path[->] (s1) edge [loop,out=100,in=140,looseness=10] node[pos=0.5,left,yshift=-5,xshift=-2] {$\frac13 a$} (s1);
	
	\path[->] (s2) edge [loop,out=80,in=40,looseness=10] node[pos=0.5,right,yshift=-5,xshift=1] {$1 a$} (s2);
	
	\path[->] (s3) edge [loop,out=200,in=160,looseness=10] node[pos=0.5,left] {$\frac12 a$} (s3);
	\path[->] (s3) edge [loop,out=20,in=340,looseness=10] node[pos=0.5,right] {$\frac12 b$} (s3);
	\end{tikzpicture}
\end{center}
The likelihood exponents are $-\infty$ and $\ell := \frac12 \ln \frac89$ so that $\Lambda = \{-\infty, \ell\}$. Also, $\PP_{s_4}(E_{-\infty}) = \PP_{s_4}(E_\ell) =\frac12$. Up to $\PP_{s_4}$-null sets the events $E_{-\infty}$, $b\Sigma^\omega$ and $ba^*b\Sigma^\omega$ are equal. The event $ba^*b\Sigma^\omega$ represents the right chain producing an observation which the left chain cannot produce, causing the SPRT to terminate for any $\alpha, \beta$.
Therefore conditioned on $E_{-\infty}$, the random variable $N_{\alpha, \beta} - 1$ is bounded by a geometric random variable with parameter $\frac12$. Hence $\sup_{\alpha, \beta} \EE_{\pi_2}~ \Big[ N_{\alpha,\beta} \mid E_{-\infty} \Big] \leq 1 + 2$.
\end{example}
We define the stopping time $N_\perp = \min\{n \in \NN \mid L_n = 0\}$. Note that $\sup_{\alpha, \beta}N_{\alpha, \beta} \leq N_\perp$ since $\{L_n = 0\} \subseteq \{L_n \leq \frac{\alpha}{1 - \beta}\}$ for all $\alpha, \beta$.
By the following proposition, the reverse inequality also holds.
\begin{restatable}{proposition}{propneginf} \label{prop:neginf}
%Up to $\PP_{\pi_2}$-null sets
The events $E_{-\infty}$ and $\{L_n = 0 \text{ for some } n\}$ are equal.
Thus, $\sup_{\alpha, \beta} N_{\alpha, \beta} = N_\perp$ and $\lim_{\alpha, \beta \rightarrow 0}\PP_{\pi_2}(N_{\alpha, \beta} < \infty) = \PP_{\pi_2}(E_{-\infty})$.
\end{restatable}
Applying this to \Cref{mortalitywaldex}, we obtain $\sup_{\alpha, \beta} \EE_{\pi_2}\big[ N_{\alpha,\beta} \mid E_{-\infty}\big] = 3$.

\subsubsection{Likelihood Exponent in $(-\infty, 0)$}
Conditioned on $E_\ell$ where $\ell \in (-\infty, 0)$, \Cref{asymptoticwald} states that $N_{\alpha, \beta}$ scales with $\frac{\ln \alpha}{\ell}$ in expectation.  The following result shows that this relationship also holds $\PP_{\pi_2}$-almost surely.
\begin{restatable}{proposition}{liexpmotivation}\label{liexpmotivation}
Let $\ell \in \Lambda$ and assume $\ell \in (-\infty, 0)$. We have
\begin{equation*}
\PP_{\pi_2}\Big( N_{\alpha,\beta} \sim \frac{\ln \alpha}{\ell} \ \ (\text{as } \alpha, \beta \rightarrow 0) \;\Big\vert\; E_{\ell}\Big) ~=~ 1.
\end{equation*}
\end{restatable}
In fact, we prove the first part of \cref{asymptoticwald} using \Cref{liexpmotivation}.
If there were a bound $M \in \NN$ such that $\PP_{\pi_2}$-a.s.\ $\frac{N_{\alpha, \beta}}{- \ln \alpha} \leq M$, the first part of \Cref{asymptoticwald} would follow from \Cref{liexpmotivation} by the dominated convergence theorem.
However this is not the case in general.
Instead we show in \cref{app:liexpmotivation} that the set of random variables $\{\frac{N_{\alpha, \beta}}{-\ln \alpha}\mid 0 < \alpha, \beta \leq \frac12\}$ is uniformly integrable with respect to the measure $\PP_{\pi_2}$ and then use Vitali's convergence theorem.

\begin{example}
Recall \Cref{sleepcycles}, where $\Lambda = \{\ell\}$.
\begin{figure}[ht]
\begin{center}
\includegraphics[width=0.65\textwidth]{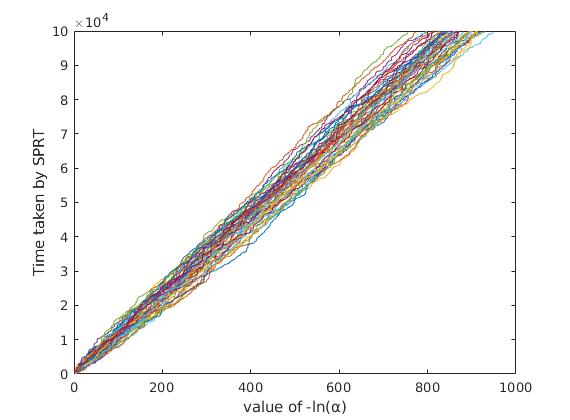}	
\end{center}
\caption{The time taken by the SPRT for $0 \leq -\ln \alpha = -\ln \beta \leq 1000$.}\label{sprttimevslnalpha}
\end{figure}
\Cref{sprttimevslnalpha} demonstrates the asymptotic relationship in \Cref{liexpmotivation}. Each of the $50$ lines correspond to a sample run and we record the value of $N_{\alpha, \beta}$ for $0 \leq - \ln \alpha = - \ln \beta \leq 1000$. From the figure we estimate $-\frac{1}{\ell}$ as $\frac{10^5}{800} = 125$. This coincides with the estimate given in \Cref{estimatingsleepcyclelikelihood}.
\end{example}

We conclude from this section that the performance of the SPRT, in terms of its termination time $N_{\alpha,\beta}$, is tightly connected to likelihood exponents.
This motivates our study of likelihood exponents in the rest of the paper.

\section{Probability of $E_\ell$} \label{sec:qual}

In this section we aim at computing $\PP_{\pi_2}(E_\ell)$ for a likelihood exponent~$\ell$.
%In particular, we want to compare this probability with $0$ or~$1$, especially for the extreme cases $\ell \in \{-\infty,0\}$.
We show the following theorem.
\begin{restatable}{theorem}{thmqualprob} \label{thm:qual-prob}
Given an HMM and initial distributions $\pi_1, \pi_2$,
\begin{enumerate}
\item one can compute $\PP_{\pi_2}(\liexp = -\infty)$ and $\PP_{\pi_2}(\liexp = 0)$ in PSPACE;
\item one can decide whether $\PP_{\pi_2}(\liexp  =0)=0$ (i.e., $0 \not\in \Lambda_{\pi_1,\pi_2}$) in polynomial time;
\item deciding whether $\PP_{\pi_2}(\liexp = 0)=1$, whether $\PP_{\pi_2}(\liexp = -\infty)=0$, and whether $\PP_{\pi_2}(\liexp  = -\infty)=1$ are all PSPACE-complete problems.
\end{enumerate}
\end{restatable}

The following example illustrates the construction underlying the PSPACE upper bound.

\begin{example} \label{ex:qual-start}
Consider another adaption of \Cref{multilimitliexp}.
\begin{center}
	\begin{tikzpicture}[scale=2.3,LMC style]
\useasboundingbox (-3,-0.4) rectangle (2,0.4);
	\node[state] (s0) at (-1.25,0) {$s_1$};
	\node[state] (s1) at (-2,0) {$s_2$};
	\node[state] (s2) at (-0.5,0) {$s_3$};
	\node[state] (s3) at (1.5,0) {$s_4$};
	\node[state] (s4) at (-2.75,0) {$s_5$};
	
	\path[->] (s0) edge node[pos=0.5,above] {$\frac14 a$} (s1);
	\path[->] (s0) edge node[pos=0.5,above] {$\frac34 b$} (s2);
	
	\path[->] (s2) edge [loop,out=-40,in=-80,looseness=10] node[pos=0.5,right,yshift=5,xshift=1] {$\frac12 b$} (s2);
	\path[->] (s2) edge [loop,out=80,in=40,looseness=10] node[pos=0.5,right,yshift=-5,xshift=1] {$\frac12 a$} (s2);
	
	\path[->] (s3) edge [loop,out=200,in=160,looseness=10] node[pos=0.5,left] {$\frac12 a$} (s3);
	\path[->] (s3) edge [loop,out=20,in=340,looseness=10] node[pos=0.5,right] {$\frac12 b$} (s3);
	
	\path[->] (s4) edge [loop,out=220,in=260,looseness=10] node[pos=0.5,left,yshift=5,xshift=-1] {$\frac13 b$} (s4);
	\path[->] (s4) edge [loop,out=100,in=140,looseness=10] node[pos=0.5,left,yshift=-5,xshift=-1] {$\frac13 a$} (s4);
	
	\path[->] (s1) edge [bend left] node[below] {$1 a$} (s4);
\path[->] (s4) edge [bend left] node[above] {$\frac13 a$} (s1);
%\DrawBoundingBox	
	\end{tikzpicture}
\end{center}
If the first letter produced by~$s_4$ is~$b$, then $L_n = \frac32$ for all $n \in \NN$.
If the first two letters are~$ab$, then $L_1 = \frac12$ and $L_n = 0$ for $n \geq 2$.
If the first two letters are~$aa$, then $s_5 \in \supp(e_{s_1} \Psi(aaw))$ for all $w \in \Sigma^*$, and therefore, up to a $\PP_{s_4}$-null set, $L_n > 0$ holds for all $n \in \NN$, which implies (using \cref{prop:neginf}) that there is $\ell \in (-\infty, 0)$ such that $\liexp = \ell$.
Thus, $\Lambda_{s_1, s_4} = \{-\infty, \ell, 0\}$.

The likelihood ratio~$L_n$ is~$0$ if and only if $\supp(\pi_1 \Psi(w_n)) = \emptyset$.
In order to track the support of $\pi_1 \Psi(w_n)$, we consider the left part of the HMM as an NFA with $s_1$ as the initial state and its determinisation as shown in the DFA below.
\begin{center}
	\begin{tikzpicture}[scale=2.3,LMC style]
	\node[state] (s0) at (-1.25,0) {$\{s_1\}$};
	\node[state] (s1) at (-2,0) {$\{s_2\}$};
	\node[state] (s2) at (-0.5,0) {$\{s_3\}$};
	\node[state] (s4) at (-2.75,0) {$\{s_5\}$};
	\node[state] (s5) at (-2.75,-0.75) {$\{s_2,s_5\}$};
	\node[state] (s6) at (-2,-0.75) {$\emptyset$};
	
	\path[->] (s6) edge [loop,out=20,in=-20,looseness=10] node[pos=0.5,right] {$a, b$} (s6);
	
	\path[->] (s0) edge node[pos=0.5,above] {$a$} (s1);
	\path[->] (s0) edge node[pos=0.5,above] {$b$} (s2);
	
	\path[->] (s2) edge [loop,out=-70,in=-110,looseness=10] node[pos=0.5,below] {$b$} (s2);
	\path[->] (s2) edge [loop,out=20,in=-20,looseness=10] node[pos=0.5,right] {$a$} (s2);
	
	\path[->] (s4) edge [loop,out=160,in=200,looseness=10] node[pos=0.5,left] {$b$} (s4);
	
	\path[->] (s1) edge node[pos=0.5,above] {$a$} (s4);
	\path[->] (s1) edge node[pos=0.5,right] {$b$} (s6);
	
	\path[->] (s4) edge [bend left] node[right] {$a$} (s5);
	\path[->] (s5) edge [bend left] node[left] {$b$} (s4);	
	
	\path[->] (s5) edge [loop,out=160,in=200,looseness=5] node[pos=0.5,left] {$a$} (s5);
%\DrawBoundingBox
	\end{tikzpicture}
\end{center}
Almost surely, $s_4$ produces a word that drives this DFA into a bottom SCC, which then determines $\liexp$: concretely, the bottom SCC $\{\{s_5\},\{s_2,s_5\}\}$ is associated with~$\ell$, the bottom SCC $\{\emptyset\}$ with~$-\infty$, and the bottom SCC $\{\{s_3\}\}$ with~$0$.
\end{example}

In general, the observations need not be produced uniformly at random but by an HMM.
Therefore, in the following construction, we also keep track of the ``current'' state of the HMM which produces the observations.
For $S \subseteq Q$ and $a \in \Sigma$, define $\delta(S,a) := \{q' \in Q \mid \exists\,q \in S: \Psi(a)_{q,q'} > 0\}$.
Define the Markov chain $\B := (2^Q \times Q, T)$ where
\[
T_{(S,q),(S',q')} \ := \ \sum_{\delta(S,a) = S'} \Psi(a)_{q,q'} \,.
\]
Given initial distributions $\pi_1, \pi_2$ on~$Q$ as before, define an initial distribution $\iota$ on $2^Q \times Q$ by $\iota((\supp(\pi_1),q)) := (\pi_2)_q$.
Intuitively, the left part~$S$ of a state $(S,q)$ tracks the support of $\pi_1 \Psi(w_n)$, and the right part~$q$ tracks the current state of the HMM that had been initialised at a random state from~$\pi_2$.
The following lemma states the key properties of this construction.

\begin{restatable}{lemma}{lemexpoprop} \label{lem:expoprop}
Consider the Markov chain $\B = (2^Q \times Q, T)$ defined above.
\begin{enumerate}
\item
Every bottom SCC of~$\B$ is associated with a single likelihood exponent; i.e., for every bottom SCC $C \subseteq 2^Q \times Q$ there is $\ell(C) \in [-\infty,0]$ such that for any initial distribution $\pi_1 \in [0,1]^Q$ and any state $q_2 \in Q$ with $(\supp(\pi_1),q_2) \in C$ we have $\Lambda_{\pi_1,e_{q_2}} = \{\ell(C)\}$.
\item
Let $(S,q) \in C$ for a bottom SCC~$C$.
If $S = \emptyset$ then $\ell(C) = -\infty$; otherwise, if $e_q$ and the uniform distribution on~$S$ are not distinguishable then $\ell(C) = 0$; otherwise $\ell(C) \in (-\infty,0)$.
\item
We have $\PP_{\pi_2}(E_\ell) = \PP_{\iota}(\{\text{visit bottom SCC $C$ with $\ell(C) = \ell$}\})$.
\end{enumerate}
\end{restatable}

All parts of the lemma rely on the observation that $\liexp$ depend only on the support of~$\pi_1$ and on the support of~$\pi_2$.
The first part of the lemma follows from L\'evy's 0-1 law.
We use this lemma for the proof of \cref{thm:qual-prob}.1.
\begin{proof}[Proof sketch for \cref{thm:qual-prob}.1]
The Markov chain~$\B$ from \cref{lem:expoprop} is exponentially big but can be constructed by a PSPACE transducer, i.e., a Turing machine whose work tape (but not necessarily its output tape) is PSPACE-bounded.
This PSPACE transducer can also identify the bottom SCCs.
For each bottom SCC~$C$, the PSPACE transducer also decides whether $\ell(C) = -\infty$ or $\ell(C) \in (-\infty,0)$ or $\ell(C) = 0$, using \cref{lem:expoprop}.2 and the polynomial-time algorithm for distinguishability from~\cite{kief14}.
Finally, to compute $\PP_{\pi_2}(E_{-\infty})$ and $\PP_{\pi_2}(E_0)$, by \cref{lem:expoprop}.3, it suffices to set up and solve a linear system of equations for computing hitting probabilities in a Markov chain.
This system can also be computed by a PSPACE transducer.
Since linear systems of equations can be solved in the complexity class NC, which is included in polylogarithmic space, one can use standard techniques for composing space-bounded transducers to compute $\PP_{\pi_2}(E_{-\infty})$ and $\PP_{\pi_2}(E_0)$ in PSPACE.
\end{proof}

\begin{proof}[Proof of \cref{thm:qual-prob}.2]
Immediate from \cref{cor:probexp0} and the polynomial-time decidability of distinguishability~\cite{kief14}.
\end{proof}

Towards a proof of \cref{thm:qual-prob}.3, we use the \emph{mortality} problem, which asks, given a finite set of states~$Q$, a finite alphabet~$\Sigma$, and a function $\Phi : \Sigma \to \{0,1\}^{Q \times Q}$, whether there exists a word $w \in \Sigma^*$ such that $\Phi(w)$ is the zero matrix.
The mortality problem can be viewed as a special case of the NFA non-universality problem (given an NFA, does it reject some word?).
Like NFA universality, the mortality problem is PSPACE-complete~\cite{karasha09}.

Concerning $\PP_{\pi_2}(E_{-\infty})$ (cf.\ \cref{thm:qual-prob}.3), we actually show a stronger result, namely that any nontrivial approximation of $\PP_{\pi_2}(E_{-\infty})$ is PSPACE-hard.
The proof is also based on the mortality problem.
\begin{proposition} \label{prop:nontrivial-approx}
There is a polynomial-time computable function that maps any instance of the mortality problem to an HMM and initial distributions $\pi_1, \pi_2$ so that if the instance is positive then $\PP_{\pi_2}(E_{-\infty})=1$ and if the instance is negative then $\PP_{\pi_2}(E_{-\infty})=0$.
Thus, any nontrivial approximation of $\PP_{\pi_2}(E_{-\infty})$ is PSPACE-hard.
\end{proposition}
\begin{proof}
Let $(Q,\Sigma,\Phi)$ be an instance of the mortality problem.
If there is $q \in Q$ that indexes a zero row in $\sum_{a \in \Sigma} \Phi(a)$, remove the row and column indexed by~$q$ in all~$\Phi(a)$.
Thus, we can assume without loss of generality that $\sum_{a \in \Sigma} \Phi(a)$ has no zero row.
Construct an HMM $(Q,\Sigma,\Psi)$ so that $\Phi(a)$ and~$\Psi(a)$ have the same zero pattern for all $a \in \Sigma$.
Define $\pi_1$ as a uniform distribution on~$Q$.
Define $\pi_2$ as a Dirac distribution on a fresh state that emits letters from~$\Sigma$ uniformly at random.
Thus, if $(Q,\Sigma,\Phi)$ is a positive instance of the mortality problem then $\PP_{\pi_2}(E_{-\infty})=1$, and if $(Q,\Sigma,\Phi)$ is a negative instance then $\PP_{\pi_2}(E_{-\infty})=0$.
\end{proof}

The proof that deciding whether $\PP_{\pi_2}(E_0) = 1$ is PSPACE-hard is similarly based on mortality.

\section{Representing Likelihood Exponents} \label{sec:rep}

In the following we show that one can efficiently represent likelihood exponents in terms of \emph{Lyapunov exponents}. 
The definition of Lyapunov exponents is based on the following definition.

\begin{definition} \label{df:lyapunov-exponent}
A \emph{matrix system} is a triple $\M = (Q, \Sigma, \Psi)$ where $Q$ is a finite set of states, $\Sigma$ is a finite set of observations, and $\Psi: \Sigma \to \RR_{\ge 0}^{Q \times Q}$ specifies the transitions.
(Note that an HMM is a matrix system.)
A \emph{Lyapunov system} is a pair $\S = (\M, \rho)$ where $\M = (Q, \Sigma, \Psi)$ is a matrix system and $\rho \in (0,1]^\Sigma$ is a probability distribution with full support, such that the directed graph $(Q,E)$ with $E = \{(q,r) \mid \sum_{a \in \Sigma} \Psi_{q,r}(a) > 0\}$ is strongly connected.
\end{definition}

%Let $((Q, \Sigma, \Psi), \rho)$ be a Lyapunov system.
We can identify the probability distribution $\rho$ from this definition with the single-state HMM $(\{s\}, \Sigma, \Psi_\rho)$ where $\Psi_\rho(a)_{s,s} = \rho(a)$ for all $a \in \Sigma$.
In this way, $\rho$ produces a random infinite word from $\Sigma^\omega$. We will write $\PP_{\rho} $ for the associated probability measure.
The following lemma is Theorem 1 from~\cite{prot13}.

\begin{lemma}[\cite{prot13}] \label{lem:lyapunov-exponent}
Let $((Q, \Sigma, \Psi), \rho)$ be a Lyapunov system.
Then there is $\lambda \in \RR$ such that, for all $\pi \in [0,\infty)^Q$, $\PP_{\rho}$-a.s., either $\pi \Psi(w_n) = \vec{0}$ for some $n \in \NN$ or the limit $\lim_{n \to \infty} \frac1n \ln \| \pi \Psi(w_n) \|$ exists and equals~$\lambda$.
\end{lemma}
For a Lyapunov system~$\S$ we call $\lambda(\S) = \lambda$ from the lemma the \emph{Lyapunov exponent} defined by~$\S$.
We prove the following theorem, which implies \cref{liexplimits}.
\begin{restatable}{theorem}{likelihoodtolyapunov}\label{thm:likelihood-to-lyapunov}
Given an HMM $(Q, \Sigma, \Psi)$ we can compute in polynomial time $2 K \le 2 |Q|^2$ Lyapunov systems $\S_1^1, \S_1^2, \S_2^1, \S_2^2, \ldots, \S_K^1, \S_K^2$ such that for any initial distributions $\pi_1, \pi_2$ the limit $\liexp$ exists $\PP_{\pi_2}$-a.s.\ and lies in
\[
\Lambda \ \subseteq \ \{-\infty\} \cup \{\lambda(\S_1^1) - \lambda(\S_1^2), \ldots, \lambda(\S_K^1) - \lambda(\S_K^2)\}\,.
\]
In particular, the HMM $(Q, \Sigma, \Psi)$ has at most $|Q|^2 + 1$ likelihood exponents.
\end{restatable}

In the rest of the section we provide more details on the construction underlying \cref{thm:likelihood-to-lyapunov}.
As an intermediate concept (between the given HMM and the Lyapunov systems from \cref{thm:likelihood-to-lyapunov}) we define \emph{generalized Lyapunov systems}.
\def\generalizedlyapsysdef{
First, for two matrix systems $\M_1 = (Q_1,\Sigma,\Psi_1)$ and $\M_2 = (Q_2,\Sigma,\Psi_2)$ with finite $Q_1, Q_2, \Sigma$ and transitions $\Psi_1,\Psi_2 : \Sigma \to \RR_{\ge 0}^{Q \times Q}$ we define the directed graph $G_{\M_1,\M_2} = (Q_1 \times Q_2, E)$ such that there is an edge from $(q_1,q_2)$ to $(r_1,r_2)$ if there is $a \in \Sigma$ with $\Psi_1(a)_{q_1,r_1} > 0$ and $\Psi_2(a)_{q_2,r_2} > 0$.

A \emph{generalized Lyapunov system} is a triple $\S = (\M,\H,C)$ where $\M = (Q_1, \Sigma, \Psi_1)$ is a matrix system and $\H = (Q_2, \Sigma, \Psi_2)$ is a strongly connected HMM and $C \subseteq Q_1 \times Q_2$ is a bottom SCC of $G_{\M, \H}$.
Given a generalized Lyapunov system, one can efficiently compute an ``equivalent'' Lyapunov system:
}
\begin{restatable}{lemma}{fromgentonongen}\label{from-gen-to-nongen}
Let $\S = ((Q_1, \Sigma, \Psi_1),(Q_2, \Sigma, \Psi_2),C)$ be a generalized Lyapunov system.
\begin{enumerate}
\item
There is $\lambda \in \RR$, henceforth called $\lambda(\S)$, such that, for all $\pi_1 \in [0,\infty)^{Q_1}$ and all probability distributions $\pi_2 \in [0,1]^{Q_2}$ with $\supp(\pi_1) \times \supp(\pi_2) \subseteq C$, we have $\PP_{\pi_2}$-a.s.\ that either $\pi_1 \Psi_1(w_n) = \vec{0}$ for some $n \in \NN$ or the limit $\lim_{n \to \infty} \frac1n \ln \| \pi_1 \Psi_1(w_n) \|$ exists and equals~$\lambda(\S)$.
\item
One can compute in polynomial time a Lyapunov system $\S'$ such that $\lambda(\S) = \lambda(\S')$.
\end{enumerate}
\end{restatable}

\def\definitionsforgenlyap{
Let $\H = (Q,\Sigma,\Psi)$ be an HMM.
Let $R \subseteq Q \times Q$ be a (not necessarily bottom) SCC of the graph $G_{\H,\H}$ such that $Q_R := \{q_2 \in Q \mid \exists\,q_1 \in Q : (q_1,q_2) \in R\}$ is a bottom SCC of the graph of $\sum_{a \in \Sigma} \Psi(a)$.
We call such~$R$ a \emph{right-bottom} SCC.
Clearly there are at most $|Q|^2$ right-bottom SCCs.
Towards \cref{thm:likelihood-to-lyapunov} we want to define, for each right-bottom SCC~$R$, two generalized Lyapunov systems $\S_R^1, \S_R^2$.
Intuitively, $\S_R^1$ and~$\S_R^2$ correspond to the numerator and the denominator of the likelihood ratio, respectively.

For a function of the form $\Phi : \Sigma \to \RR^{Q \times Q}$ and $P \subseteq Q$ we write $\Phi_{|P} : \Sigma \to \RR^{P \times P}$ for the function with $\Phi_{|P}(a)(q,r) = \Phi(a)(q,r)$ for all $a \in \Sigma$ and $q,r \in P$; i.e., $\Phi_{|P}(a)$ denotes the principal submatrix obtained from~$\Phi(a)$ by restricting it to the rows and columns indexed by~$P$.

Define $\Psi'(a,r)_{q,r} := \Psi(a)_{q,r}$ for all $a \in \Sigma$ and $q,r \in Q$.
Then $(Q, \Sigma \times Q, \Psi')$ is an HMM, which is similar to~$\H$, but which emits, in addition to an observation from~$\Sigma$, also the next state.
Since $Q_R$ is a bottom SCC of the graph of $\sum_{a \in \Sigma} \Psi(a)$, the HMM $\H_2 := (Q_R, \Sigma \times Q_R, \Psi'_{|Q_R})$ is strongly connected.
This HMM~$\H_2$ will be used both in $\S_R^1$ and in~$\S_R^2$.

Next, define $\overline{\Psi} : (\Sigma \times Q) \to [0,1]^{(Q \times Q) \times (Q \times Q)}$ by
\[
 \overline{\Psi}(a,r_2)_{(q_1,q_2),(r_1,r_2)} \ := \ \Psi(a)_{q_1,r_1} \quad \text{for all } a \in \Sigma\ \text{ and } q_1, q_2, r_1, r_2 \in Q\,.
\]
Now define $\S_R^1 := (\M^1, \H_2, C^1)$, where $\M^1 := (R, \Sigma \times Q_R, \overline{\Psi}_{|R})$ %and
%\[
% \Psi^1(a,r_2)_{(q_1,q_2),(r_1,r_2)} \ := \ \Psi(a)_{q_1,r_1} \quad \text{for all } (a,r_2) \in \Sigma \times Q_2\ \text{ and }\ (q_1,q_2), (r_1,r_2) \in R
%\]
and $C^1 := \{((q_1,q_2),q_2) \mid (q_1,q_2) \in R\}$.
Finally, denoting by $R' \subseteq Q_R \times Q_R$ the SCC of the graph $G_{\H,\H}$ that contains the ``diagonal'' vertices $(q,q) \in Q_R \times Q_R$, define $\S_R^2 := (\M^2, \H_2, C^2)$, where $\M^2 := (R', \Sigma \times Q_R, \overline{\Psi}_{|R'})$
and $C^2 := \{((q_1,q_2),q_2) \mid (q_1,q_2) \in R'\}$.

For sets $U, V \subseteq Q \times Q$ let $U \longrightarrow_{G_{\H,\H}} V$ denote that there are $u \in U$ and $v \in V$ such that $v$ is reachable from~$u$ in~$G_{\H,\H}$.
}

\definitionsforgenlyap

We are ready to state the following key technical lemma:%
%\begin{lemma} \label{lem:construct-L-systems}
%Given an HMM $(Q, \Sigma, \Psi)$, let $\mathcal{R} \subseteq 2^{Q \times Q}$ be the set of its right-bottom SCCs, and, for $R \in \mathcal{R}$, let $\S_{R}^1, \S_{R}^2$ be the generalized Lyapunov systems defined above.
%Then,
%\[
%\Lambda \ \subseteq \ \{-\infty\} \cup \{\lambda(\S_{R}^1) - \lambda(\S_{R}^2) \mid R \in \mathcal{R}\} \,.
%\]
%\end{lemma}
%NEW AND STRONGER VERSION:
\begin{restatable}{lemma}{constructLsystems}\label{lem:construct-L-systems}
Given an HMM $(Q, \Sigma, \Psi)$, let $\mathcal{R} \subseteq 2^{Q \times Q}$ be the set of its right-bottom SCCs, and, for $R \in \mathcal{R}$, let $\S_{R}^1, \S_{R}^2$ be the generalized Lyapunov systems defined above.
Then, for any initial distributions $\pi_1, \pi_2$, the limit $\liexp$ exists $\PP_{\pi_2}$-a.s.\ and lies in
\[
\{-\infty\} \cup \{\lambda(\S_{R}^1) - \lambda(\S_{R}^2) \mid R \in \mathcal{R},\ \supp(\pi_1) \times \supp(\pi_2) \longrightarrow_{G_{\H,\H}} R\} \,.
\]
Thus,
$\Lambda_{\pi_1,\pi_2} \subseteq \{-\infty\} \cup \{\lambda(\S_{R}^1) - \lambda(\S_{R}^2) \mid R \in \mathcal{R},\ \supp(\pi_1) \times \supp(\pi_2) \longrightarrow_{G_{\H,\H}} R\}$.
\end{restatable}

\begin{proof}[Proof sketch]
Let $\pi_1, \pi_2$ be initial distributions.
%Recall the definition of likelihood ratio $L_n(w) = \frac{\| \pi_1 \Psi(w_n) \|}{\| \pi_2 \Psi(w_n) \|}$.
Very loosely speaking, we show in the appendix that on $\PP_{\pi_2}$-almost every run~$w$ there is a right-bottom SCC~$R$ which ``traps'' ``most'' of the mass of $\pi_1 \Psi(w_n)$ and $\pi_2 \Psi(w_n)$.
This can be made meaningful and formal using (the cross-product systems) $\S_R^1, \S_R^2$.
We then show that on $\PP_{\pi_2}$-almost every such run~$w$, for both $i=1,2$, the limit $\lim_{n \to \infty} \frac{1}{n} \ln \| \pi_i \Psi(w_n) \|$ exists and equals $\lambda(\S_R^i)$ (or $\pi_1 \Psi(w_n) = \vec{0}$ for some~$n$).
It follows that
\[
 \liexp \ = \ \lim_{n \to \infty} \frac{1}{n} \ln \frac{\| \pi_1 \Psi(w_n) \|}{\| \pi_2 \Psi(w_n) \|} \ = \ \lambda(\S_R^1) - \lambda(\S_R^2) \,. \qedhere
\]
\end{proof}
With \cref{lem:construct-L-systems} at hand, the proof of \cref{thm:likelihood-to-lyapunov} is easy:
\begin{proof}[Proof of \cref{thm:likelihood-to-lyapunov}]
As argued before, the set $\mathcal{R}$ of right-bottom SCCs of the given HMM has at most $|Q|^2$ elements.
These right-bottom SCCs $R$ and the associated generalized Lyapunov systems $\S_R^1, \S_R^2$ can be computed in polynomial time.
By \cref{lem:construct-L-systems} we have $\Lambda = \bigcup_{\pi_1,\pi_2} \Lambda_{\pi_1,\pi_2} \subseteq \{-\infty\} \cup \{\lambda(\S_{R}^1) - \lambda(\S_{R}^2) \mid R \in \mathcal{R}\}$.
By \cref{from-gen-to-nongen}.2, for each $R \in \mathcal{R}$ one can compute in polynomial time an equivalent Lyapunov system.
\end{proof}

\Cref{thm:likelihood-to-lyapunov} allows us to represent the likelihood exponents of an HMM in terms of Lyapunov exponents.
In general, approximating or even computing Lyapunov exponents is hard, but there are practical approximation algorithms using convex optimisation~\cite{ProtasovJungers13,Sutter21}.

\section{Deterministic HMMs} \label{sec:det}

In \cref{sec:qual,sec:rep} we have seen that the problems of representing/computing likelihood exponents and of computing their probabilities tend to be computationally difficult.
In this section we study \emph{deterministic} HMMs and show that this subclass leads to tractable problems.
An HMM $(Q, \Sigma, \Psi)$ is \emph{deterministic} if, for all $a \in \Sigma$, all rows of $\Psi(a)$ contain at most one non-zero entry.
%Let $\H = (Q, \Sigma, \Psi)$ be a deterministic HMM for the rest of the section.
Thus, for all $q \in Q$ and $w \in \Sigma^*$, we have $|\supp(e_q \Psi(w))| \le 1$.

A useful observation is that the Markov chain $\B = (2^Q \times Q, T)$, which was defined before \cref{lem:expoprop} and can be exponential in general, has only quadratic size in the deterministic case if we restrict it to the part that is reachable from initial Dirac distributions.

\begin{example}\label{ex:det1}
Consider the deterministic HMM $(Q, \Sigma, \Psi)$ in \cref{fig:det1}(a).
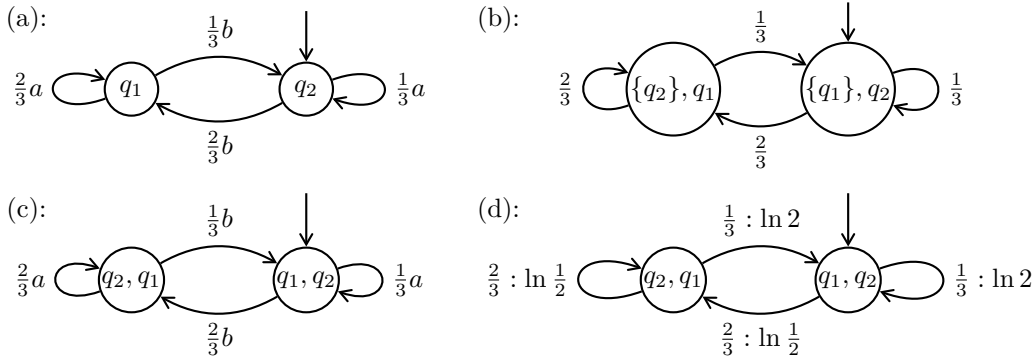
\begin{figure}[ht]
\begin{center}
	\begin{tikzpicture}[scale=2.3,LMC style]
    \node at (-1.1,0.4) {(a):};

	\node[state] (s01) at (-0.5,0) {$q_1$};
	\node[state] (s11) at (0.5,0) {$q_2$};
    \draw[->,thick] (0.5,0.45) -- (s11);	
	\path[->] (s01) edge [loop,out=200,in=160,looseness=10] node[pos=0.5,left] {$\frac23 a$} (s01);
	\path[->] (s01) edge [out=20,in=170,bend left] node[pos=0.5,above] {$\frac13 b$} (s11);
	\path[->] (s11) edge [out=200,in=340,bend left] node[pos=0.5,below] {$\frac23 b$} (s01);
	\path[->] (s11) edge [loop,out=20,in=340,looseness=10] node[pos=0.5,right] {$\frac13 a$} (s11);

    \node at (1.6,0.4) {(b):};

	\node[state] (s02) at (2.6,0) {$\{q_2\}, q_1$};
	\node[state] (s12) at (3.6,0) {$\{q_1\}, q_2$};	
    \draw[->,thick] (3.6,0.5) -- (s12);	
	\path[->] (s02) edge [loop,out=200,in=160,looseness=5] node[pos=0.5,left] {$\frac23$} (s02);
	\path[->] (s02) edge [out=20,in=170,bend left] node[pos=0.5,above] {$\frac13$} (s12);
	\path[->] (s12) edge [out=200,in=340,bend left] node[pos=0.5,below] {$\frac23$} (s02);
	\path[->] (s12) edge [loop,out=20,in=340,looseness=5] node[pos=0.5,right] {$\frac13$} (s12);

    \node at (-1.1,-0.7) {(c):};

	\node[state] (s03) at (-0.5,-1.1) {$q_2, q_1$};
	\node[state] (s13) at (0.5,-1.1) {$q_1, q_2$};
    \draw[->,thick] (0.5,-0.6) -- (s13);	
	\path[->] (s03) edge [loop,out=200,in=160,looseness=7] node[pos=0.5,left] {$\frac23 a$} (s03);
	\path[->] (s03) edge [out=20,in=170,bend left] node[pos=0.5,above] {$\frac13 b$} (s13);
	\path[->] (s13) edge [out=200,in=340,bend left] node[pos=0.5,below] {$\frac23 b$} (s03);
	\path[->] (s13) edge [loop,out=20,in=340,looseness=7] node[pos=0.5,right] {$\frac13 a$} (s13);

    \node at (1.6,-0.7) {(d):};

	\node[state] (s04) at (2.6,-1.1) {$q_2, q_1$};
	\node[state] (s14) at (3.6,-1.1) {$q_1, q_2$};
    \draw[->,thick] (3.6,-0.6) -- (s14);	
	\path[->] (s04) edge [loop,out=200,in=160,looseness=10] node[pos=0.5,left] {$\frac23 : \ln \frac12$} (s04);
	\path[->] (s04) edge [out=20,in=170,bend left] node[pos=0.5,above] {$\frac13 : \ln 2$} (s14);
	\path[->] (s14) edge [out=200,in=340,bend left] node[pos=0.5,below] {$\frac23 : \ln \frac12$} (s04);
	\path[->] (s14) edge [loop,out=20,in=340,looseness=10] node[pos=0.5,right] {$\frac13 : \ln 2$} (s14);
%\DrawBoundingBox
\end{tikzpicture}
\end{center}
\caption{Cross-product constructions for a deterministic HMM.}
\label{fig:det1}
\end{figure}
Let $\pi_1 = e_{q_1}$ and $\pi_2 = e_{q_2}$ (the latter is indicated by an arrow pointing to~$q_2$).
Then the relevant (i.e., reachable from $(\{q_1\}, q_2)$) part of~$\B$ is shown in \cref{fig:det1}(b).
Let us add back the observations that gave rise to the transitions in~$\B$, and for simplicity drop the set brackets in the left component of states.
We obtain the HMM in \cref{fig:det1}(c).
With this HMM we may keep track of the exact likelihood ratio.
For example, suppose that the word $a b a$ is emitted, so that $L_3 = \frac{\|e_{q_1} \Psi(a b a)\|}{\|e_{q_2} \Psi(a b a)\|} = \frac12$ and $\supp(e_{q_1} \Psi(a b a)) = \{q_2\}$ and $\supp(e_{q_2} \Psi(a b a)) = \{q_1\}$.
Suppose the next letter is $b$ (which is the case with probability~$\frac13$).
Then $L_4$ arises from $L_3$ by multiplying with $\frac{\Psi_{q_2,q_1}(b)}{\Psi_{q_1,q_2}(b)} = 2$, and the supports are switched again.
In terms of log-likelihoods, we have $\ln L_4 = \ln L_3 + \ln 2$.
This motivates the Markov chain shown in \cref{fig:det1}(d), where the transitions outgoing from a state $(r_1,r_2)$ are labelled by the log-likelihood ratio of their corresponding probabilities in the HMM.
%Suppose two copies of the chain are started from states $q_1$ and $q_2$ respectively, if the chains emit the same word, they will be in opposite states.
%Let the chain started from $q_1$ emit word $aab \in \Sigma^3$. The chain will now be in state $q_2$. Suppose the next letter emitted is $b$. The value of $L_4 = L_3 \times \frac23 / \frac13 = 2 L_3$ and $\ln L_4 = \ln L_3 + \ln 2$. In general deterministic chains have the property that $\ln L_n = R_n + \ln L_{n - 1}$ where $R_n$ is a random variable depending only the transition chosen. In the example above, we can build a Markov chain to characterise the joint distribution of states. We are interested in $\liexp$ with words produced from starting state $q_2$ so the transition probabilities are lifted from the second chain. In the diagram below we also include the value of $R_n$ on each transition.
%\begin{center}
%	\begin{tikzpicture}[scale=2.3,LMC style]
%	\node[state] (s0) at (-0.5,0) {$q_2, q_1$};
%	\node[state] (s1) at (0.5,0) {$q_1, q_2$};
%	
%	\path[->] (s0) edge [loop,out=200,in=160,looseness=10] node[pos=0.5,left] {$\frac23 : \ln \frac12$} (s0);
%	\path[->] (s0) edge [out=20,in=170,bend left] node[pos=0.5,above] {$\frac13 : \ln 2$} (s1);
%	\path[->] (s1) edge [out=200,in=340,bend left] node[pos=0.5,below] {$\frac23 : \ln \frac12$} (s0);
%	\path[->] (s1) edge [loop,out=20,in=340,looseness=10] node[pos=0.5,right] {$\frac13 : \ln 2$} (s1);
%	\end{tikzpicture}.
%\end{center}
The Markov chain has stationary distribution $(\frac23, \frac13)$.
By the strong ergodic theorem for Markov chains, we obtain (the irrational number)
\[
\textstyle
\liexp \ =\ \frac23 \Big(\frac23 \ln \frac12 + \frac13 \ln 2 \Big) + \frac13 \Big(\frac13 \ln 2 + \frac23 \ln \frac12\Big) \ = \ \frac13 \ln2 + \frac23 \ln \frac12 = - \frac13 \ln2\,.
\]
%This shows that likelihood exponents can be irrational.
\end{example}
In general there may again be several likelihood exponents, including $-\infty$ and~$0$.
%Moreover, these limits may not be rational numbers but linear combinations of natural logarithms of rationals.
For the rest of the section, let $\H = (Q,\Sigma,\Psi)$ be a deterministic HMM. % and $\pi_1 = e_{q_1}$ and $\pi_2 = e_{q_2}$ initial distributions.
%For $q \in Q$ and $a \in \Sigma$, if there is $q' \in Q$ with $\Psi(a)_{q,q'} > 0$ define $\delta(q,a) := q'$, otherwise set $\delta(q,a) := \bot$.
Motivated by \cref{ex:det1}, define an HMM $\A = ((Q \times Q) \cup s_\bot, \hat\Sigma, \hat\Psi)$, where $s_\bot$ is a fresh state, and
\begin{align*}
&\hat\Sigma \ := \ \left\{\ln \frac{\Psi(a)_{q_1,r_1}}{\Psi(a)_{q_2,r_2}} \in [-\infty,\infty) \;\middle\vert\; a \in \Sigma,\ q_1,r_1,q_2,r_2 \in Q,\ \Psi(a)_{q_2,r_2} \ne 0 \right\} \cup \{-\infty\} \\
&\hat\Psi(\hat a)_{(q_1,q_2),(r_1,r_2)} \ := \ \sum \left\{\Psi(a)_{q_2,r_2} \;\middle\vert\; a \in \Sigma : \hat a = \ln \frac{\Psi(a)_{q_1,r_1}}{\Psi(a)_{q_2,r_2}} \right\} \quad \text{for } \hat a \ne -\infty \\
&\hat\Psi(-\infty)_{(q_1,q_2),s_\bot} \ := \ \sum \left\{\Psi(a)_{q_2,r_2} \;\middle\vert\; a \in \Sigma,\ r_2 \in Q : \textstyle\sum_{r_1 \in Q} \Psi(a)_{q_1,r_1} = 0 \right\}  \\
&\hat\Psi(-\infty)_{s_\bot,s_\bot} \ := \ 1\,.
\end{align*}

Note that the embedded Markov chain of~$\A$ is similar to the Markov chain~$\B$ from \cref{lem:expoprop}:
states $(\{q_1\}, q_2)$ in~$\B$ are called $(q_1,q_2)$ in~$\A$, the states $(\emptyset,q)$ in~$\B$ are subsumed by the state $s_\bot$ of~$\A$, and the states $(S,q)$ in~$\B$ with $|S| > 1$ are not represented in~$\A$.
The observations in~$\hat\Sigma \subseteq [-\infty, \infty)$ track the log-likelihood ratio.

\begin{example} \label{ex:det2}
Consider the HMM~$\H$ on the left, with initial distributions $\pi_1 = e_{q_1}$ and $\pi_2 = e_{q_2}$.
The part of~$\A$ reachable from $(q_1,q_2)$ is shown on the right:
\begin{center}
\begin{tikzpicture}[scale=2.3,LMC style]
\useasboundingbox (-1.25,-0.35) rectangle (4.5,0.35);
	\node[state] (q1) at (-1,0) {$q_1$};
	\node[state] (q2) at (-0.3,0) {$q_2$};
	\path[->] (q1) edge node[above] {$1 a$} (q2);
	\path[->] (q2) edge [loop,out=40,in=80,looseness=10] node[pos=0.5,right,yshift=-5] {$\frac12 a$} (q2);
	\path[->] (q2) edge [loop,out=-40,in=-80,looseness=10] node[pos=0.5,right,yshift=6] {$\frac12 b$} (q2);
	\node[state] (sb) at (1.3,0) {$s_\bot$};
	\node[state] (q1q2) at (2.5,0) {$q_1,q_2$};
	\node[state] (q2q2) at (3.7,0) {$q_2,q_2$};
    \draw[->,thick] (2.5,0.4) -- (q1q2);
	\path[->] (sb) edge [loop,out=160,in=200,looseness=10] node[pos=0.5,left] {$1 : -\infty$} (sb);
    \path[->] (q1q2) edge node[above] {$\frac12 : -\infty$} (sb);
    \path[->] (q1q2) edge node[above] {$\frac12 : \ln 2$} (q2q2);
	\path[->] (q2q2) edge [loop,out=20,in=-20,looseness=7] node[pos=0.5,right] {$1 : 0$} (q2q2);
%\DrawBoundingBox
\end{tikzpicture}
\end{center}
Here we have $\Lambda_{\pi_1,\pi_2} = \{-\infty,0\}$ with $\PP_{\pi_2}(E_{-\infty}) = \PP_{\pi_2}(E_0) = \frac12$.
\end{example}

Denote by $\oA$ the embedded Markov chain of~$\A$.
Let $C \subseteq Q \times Q$ be a non-$\{s_\bot\}$ bottom SCC of~$\oA$.
Let $\mu \in [0,1]^C$ denote the stationary distribution of the restriction of~$\oA$ on~$C$.
Define the vector $\nu \in \RR^C$ of average observations by
$
 \nu_{(r_1,r_2)} \ := \ \sum_{\hat a \in \hat\Sigma} \| e_{(r_1,r_2)} \hat\Psi(\hat a) \| \cdot \hat a
$.
By the strong ergodic theorem for Markov chains, the \emph{average observation} in~$C$ equals $\mu \nu^\top =: \ell(C)$.
Extend this definition by $\ell(\{s_\bot\}) := -\infty$.
%The following lemma is analogous to \cref{lem:expoprop} but simpler.
Then we have the following lemma.

\begin{restatable}{lemma}{lempolyprop} \label{lem:polyprop}
Let $\pi_1 = e_{q_1}$ and $\pi_2 = e_{q_2}$ be initial distributions.
For the Markov chain~$\oA$ define $\iota := e_{(q_1,q_2)}$.
We have
$\PP_{\pi_2}(E_\ell) = \PP_{\iota}(\{\text{visit bottom SCC $C$ with $\ell(C) = \ell$}\})$.
\end{restatable}
The proof is essentially the same as in \cref{lem:expoprop}.3.
This gives us the following result.
\begin{restatable}{theorem}{thmdet} \label{thm:det}
Given a deterministic HMM $(Q, \Sigma, \Psi)$ with initial Dirac distributions $\pi_1, \pi_2$, one can compute in polynomial time
\begin{enumerate}
\item $\Lambda_{\pi_1,\pi_2}$ as a set of expressions of the form $\sum_i x_i \ln y_i$ where $x_i,y_i \in \QQ$, and
\item $\Pr_{\pi_2}(E_\ell)$ for each such $\ell \in \Lambda_{\pi_1,\pi_2}$.
\end{enumerate}
\end{restatable}
\begin{proof}[Proof sketch]
%In a Markov chain one can compute the stationary distribution and hitting probabilities in polynomial time.
The theorem follows mostly from \cref{lem:polyprop}, with the slight complication that for part~2 we have to check numbers of the form $\sum_i x_i \ln y_i$ (where $x_i,y_i \in \QQ$) for equality.
But this can be done in polynomial time as shown in~\cite{EtessamiSY14}.
\end{proof}

\section{Conclusions} \label{sec:conclusions}
We have shown that the performance of the SPRT is tightly connected with likelihood exponents.
These numbers are related to Lyapunov exponents and can be viewed as a distance measure between HMMs.
We have shown that the number of likelihood exponents is quadratic in the number of states.
The associated computational problems tend to be complex (PSPACE-hard), but become tractable for deterministic HMMs.
In our work we did not make any ergodicity assumptions on the HMMs, unlike in earlier works from mathematics and engineering such as \cite{JuangRabiner85,ChenWillett00,Fuh03,GrossiLops08}.
Efficient approximation of likelihood exponents, in theory or praxis, remains an open problem.

\bibliography{lyapunovhmm}

\appendix

\section{Proofs and Additional Material on \cref{sec:prelims}} \label{app:prelims}

\subsection{Proof of \cref{convergenceLn}}

\convergenceLn*

\begin{proof}
The first part is \cite[Proposition~6]{kief14}.
Towards the second part, the following equalities hold.
\stefan{Define notation $\land$}
\begin{align*}
1 - d(\pi_1, \pi_2) & = \lim_{n \rightarrow \infty} \sum_{w \in \Sigma^n} \min\{\| \pi_1 \Psi(w) \|, \| \pi_2 \Psi(w) \|\} && \text{by \cite[Theorem~ 7]{kief14}}\\
& = \lim_{n \rightarrow \infty} \sum_{w \in \Sigma^n} \min\{L_n(w), 1 \}\| \pi_2 \Psi(w)\| \\
& = \lim_{n \rightarrow \infty} \EE_{\pi_2}\big[\min\{ L_n, 1\}  \big] \\
& = \EE_{\pi_2} \big[ \lim_{n \rightarrow \infty}  \min\{L_n, 1\}  \big] && \text{as } 0 \leq \min\{L_n(w), 1\} \leq 1.
\end{align*}
Then, $\lim_{n \rightarrow \infty}  \min\{L_n, 1\} = 0 \iff \lim_{n \rightarrow \infty}  L_n = 0$.
\end{proof}

\subsection{Details on \cref{sleepcycles}} \label{app:sleepcycles}

In \cite{rockhart13} they derived two embedded Markov chains with the following transition matrices:
\begin{equation*}T_1 = \begin{bmatrix}
0.793 & 0.099 & 0.035 & 0.064 &	0.009 \\
0.078 & 0.769 & 0.006 & 0.144 & 0.003 \\
0.018 & 0.004 & 0.833 & 0.134 & 0.012  \\
0.022 & 0.094 & 0.054 & 0.827 & 0.002 \\
0.011 & 0.005 & 0.035 & 0.005 & 0.945  \\
\end{bmatrix}, T_2 = \begin{bmatrix}
0.641 & 0.109 & 0.031 & 0.040 & 0.015 \\
0.202 & 0.699 & 0.008 & 0.089 & 0.003 \\
0.026 & 0.002 & 0.823 & 0.062 & 0.035 \\
0.123 & 0.189 & 0.114 & 0.808 & 0.016 \\
0.007 & 0.001 & 0.024 & 0.001 & 0.931 \\
\end{bmatrix}.
\end{equation*}

Their HMMs are state-labelled.
For each state $i$, they fit a Dirichlet probability density function (pdf) $f_i$ describing the distribution of observations in $\Delta^3$ emitted at state $i$.
The pdfs of diseased and healthy individuals were so similar that they used the same pdf for both HMMs. % which was estimated from the whole population.
Thus the two HMMs differ only in the transition probabilities.

Since $\Delta^3$ is infinite and in this paper we assume finite observation alphabets, we partition the simplex into the sets
\[
 U_k = \{x \in \Delta^3 \mid f_k(x) \geq \sup_{i} f_i(x)\}
\]
for $k = 1, \dots, 5$.
The set $U_k$ contains the points in $\Delta^3$ most likely to be produced in state~$k$.
We assign a letter $a_k$ for each $U_k$, and define a set of observations $\Sigma = \{a_1, \dots, a_5\}$.
Thus, the probability of producing letter $a_k$ from state $i$ is given as $O_{i,k} = \int_{U_k} f_i(x)\, dx$. We estimated the entries of $O$ using a numerical Monte Carlo technique. We generated 100,000 samples from all 5 Dirichlet distributions in their paper which yielded the estimate
\begin{equation*}
O = \begin{pmatrix}
0.9172&0.0803&0&0.0002&0.0024\\
0.0719&0.8606&0&0.0665&0.0010\\
0&0.0007&0.8546&0.1055&0.0392\\
0.0008&0.0998&0.0663&0.8257&0.0075\\
0.0109&0.0094&0.1046&0.0334&0.8416\\
\end{pmatrix}.
\end{equation*}
Since we consider transition labelled HMMs, we define transition functions $\Psi_1, \Psi_2$ with
\[
 \Psi_m(a_k)_{i,j} = \big( T_m \big)_{i,j} O_{i,k}
\] for $m = 1, 2$.
Let $Q = [10]$. We construct the HMM $(Q, \Sigma, \Psi)$ where
\begin{equation*}
\Psi(a) = \begin{pmatrix}
\Psi_1(a) & 0 \\
0 & \Psi_2(a) \\
\end{pmatrix}
\end{equation*}
for each $a \in \Sigma$.

Let $\pi_1$ and $\pi_2$ be the Dirac distributions on states 1 and 6 respectively. These initial distributions correspond to healthy and diseased individuals started from sleep state 1.

\section{Proofs from \Cref{liexpsubsect}}

\subsection{Proof of \cref{sprtcorrectness}}

\sprtcorrectness*

\begin{proof}
We wish to control the probabilities $\PP_{\pi_2}\big( L_N > B\big)$ and $\PP_{\pi_1}\big( L_N < A\big)$ by choosing suitable values of $A$ and $B$. Write $N := N_{\alpha, \beta}$ and let $W_n^1 = \{ w \in \Sigma^\omega \mid  A \leq L_m(w) \leq B ~\forall m < n, L_n < A\}$ then
\begin{align*}
\PP_{\pi_1}\big( L_N < A \big) & = \sum_{n = 1}^\infty \PP_{\pi_1}\big( W_n^1 \big) = \sum_{n = 1}^\infty \sum_{w \in W_n^1} \pi_1 \Psi(w) \1^T = \sum_{n = 1}^\infty \sum_{w \in W_n^1} L_n(w) \pi_2 \Psi(w) \1^T \\
& \leq A \sum_{n = 1}^\infty  \sum_{w \in W_n^1} \pi_2 \Psi(w) \1^T = A \sum_{n = 1}^\infty  \PP_{\pi_2}\big( W_n^1 \big) = A \PP_{\pi_2}\big( L_N < A \big). \\
\end{align*}
Similarly, we may derive $\PP_{\pi_2}\big( L_N > b \big) \geq \frac{1}{b} \PP_{\pi_1}\big( L_N > b\big)$ so it follows that
\begin{align*}
A & \geq \frac{\PP_{\pi_1}\big(  L_N < A\big)}{\PP_{\pi_2}\big(  L_N < A\big)} = \frac{\PP_{\pi_1}\big(  L_N < A\big)}{1 - \PP_{\pi_2}\big(  L_N > B\big)} \\
B & \leq \frac{\PP_{\pi_1}\big(  L_N > B\big)}{\PP_{\pi_2}\big(  L_{N} > B\big)} = \frac{1 - \PP_{\pi_1}\big( L_N < A\big)}{\PP_{\pi_2}\big(  L_N > B\big)}\\
\end{align*}
to guarantee the error bounds $\alpha = \PP_{\pi_1}\big( L_N < A\big)$ and $\beta = \PP_{\pi_2}\big( L_N > B\big)$.
\end{proof}

\subsection{Proof of \Cref{asymptoticwald}}

\asymptoticwald*

We will prove \Cref{asymptoticwald} later using results in this section

\probexpzero*
Towards the proof of \Cref{probexp0} we use the following which is Theorem 5 from \cite{kief16}.
\begin{lemma}\label{kief16thm5}
Let $(Q, \Sigma, \Psi)$ be an HMM and let $\pi_1$ and $\pi_2$ be initial distributions. If $\pi_1$ and $\pi_2$ are distinguishable then there is $c > 0$ such that
\begin{equation*}
\PP_{\pi_2}\Big( L_{2|Q|n} \leq 1 \Big) - \PP_{\pi_1}\Big( L_{2|Q|n} \geq 1 ) \Big) \geq 1 - 2\exp \big(-\frac{c^2}{18}n\big).
\end{equation*}
\end{lemma}

\begin{proof}[Proof of \Cref{probexp0}]
By \cref{lem:expoprop} there are a set of bottom SCCs $\mathcal{Z}$ in $\mathcal{B}$. Such that for all $Z \in \mathcal{Z}$ we have $\ell(Z) = \{0\}$. Let $\pi \in [0,1]^Q$ and $r \in Q$ such that $(\supp~\pi, r) \in Z$. Suppose that $\pi$ and $\delta_r$ are distinguishable then by \Cref{kief16thm5} both $\PP_{\delta_r}(L_n^* \geq 1) \leq 2\exp\big( -\frac{c^2}{18}n \big)$ and $\PP_{\pi}(L_n^* \leq 1) \leq 2\exp\big( -\frac{c^2}{18}n \big)$ where $L_n^*$ is the likelihood ratio started from initial distributions $\pi$ and $\delta_r$. Fix $-\frac{c^2}{18} < \alpha \leq 0$ and define the event $W_n = \{1 > L_n^* \geq \e^{n\alpha}\}$. Then
\begin{align*}
\PP_{\delta_r}(\lim_{n \rightarrow \infty} \frac1n \ln L_n^* > \alpha) & \leq \PP_{\delta_r}(\liminf_n \{\frac1n \ln L_n^* \geq \alpha\}) \\
& \leq \liminf_n \PP_{\delta_r}(\frac1n \ln L_n^* \geq \alpha) \\
& \leq \liminf_n \PP_{\delta_r}(L_n^* \geq \e^{n\alpha}) \\
& = \liminf_n \Big[ \PP_{\delta_r}(1 > L_n^* \geq \e^{n\alpha}) +  \PP_{\pi_2}(L_n^* \geq 1) \Big]\\
& \leq \liminf_n \Big[ \sum_{w \in W_n} \delta_r \Psi(w) \1^T + 2\exp\big( -\frac{c^2}{18}n \big) \Big]\\
& \leq \liminf_n \Big[ e^{-n\alpha} \sum_{w \in W_n} \pi \Psi(w) \1^T\Big]\\
& \leq \liminf_n \Big[ e^{-n\alpha} \PP_{\pi}(L_n^* < 1)\Big]\\
& \leq \liminf_n \Big[ 2e^{-n\alpha} \exp\big( -\frac{c^2}{18}n \big)\Big]\\
& = 0.
\end{align*}
In particular, $\PP_{\pi_2}(\lim_{n \rightarrow \infty} \frac1n \ln L_n = 0 ) = 0$ which contradicts $\Lambda = \{0\}$. Hence $\pi$ and $\delta_r$ are not distinguishable and so $\PP_{\delta_r}$-almost surely, we have $\lim_{n \rightarrow \infty} L_n^* > 0$. By conditioning on the events $\{a_1 r_1 \cdots a_n r_n \in (\Sigma Q)^* \mid \supp~\pi_1 \Psi(w) = \supp~ \pi, r_n = r\}$ it follows that $E_0 = \{\lim_{n \rightarrow \infty} L_n > 0\}$. We now show the second equality. If $\lim_{n \rightarrow \infty} L_n > 0$ then for $\alpha, \beta$ small enough $L_n$ never crosses the SPRT bounds. Hence, we have $\{\lim_{n \rightarrow \infty} L_n > 0\} \subseteq \bigcup_{\alpha, \beta} \{N_{\alpha, \beta} = \infty\}$. For the converse inclusion, suppose that $N_{\alpha, \beta} = \infty$ for some $\alpha, \beta$ this would contradict $\lim_{n \rightarrow \infty} L_n = 0$ since then $N_{\alpha, \beta}$ would be $\PP_{\pi_2}$-almost surely finite.
\end{proof}

\subsection{Proof of \cref{prop:neginf}}

\propneginf*
\begin{proof}
The right-to-left inclusion is clear.
\newcommand{\pmin}{p_{\mathit{min}}}
Towards the converse, let $\pmin>0$ be the minimum non-zero entry in~$\pi_1$ and all $\Psi(a)$ where $a \in \Sigma$.
Suppose that $L_n>0$ holds for all~$n$.
Then we have for all $n \ge 1$:
\begin{align*}
 \frac1n \ln L_n \ &=\ \frac1n \ln \frac{\| \pi_1 \Psi(w_n) \|}{\| \pi_2 \Psi(w_n) \|} \ \ge\ \frac1n  \ln \| \pi_1 \Psi(w_n) \| \ \ge\ \frac1n \ln \pmin^{n+1} \ = \ \frac{n+1}{n} \ln \pmin \\
                   &\ge\ 2 \ln \pmin\,.
\end{align*}
Thus, $\liexp \ne -\infty$. We have $\sup_{\alpha, \beta} N_{\alpha, \beta} \leq N_\perp$. Also,
\begin{equation*}
\bigcap_{\alpha, \beta} \{L_n \not\in (\frac{\alpha}{1-\beta},\frac{1-\alpha}{\beta})\} = \{L_n = 0\}
\end{equation*}
for all $n \in \NN$ and so $\sup_{\alpha, \beta} N_{\alpha, \beta} = N_\perp$. The final claim follows because $\{N_\perp < \infty\} = E_{-\infty}$.
\stefan{The things about $N_{\alpha, \beta}$ might require justification.}
\end{proof}

\subsection{Proof of \cref{liexpmotivation}} \label{app:liexpmotivation}

Towards the proof of \cref{liexpmotivation} we first show the following lemma.

\begin{lemma}\label{unifintegofN}
The set of random variables $\{\frac{N_{\alpha, \beta}}{-\ln \alpha}\mid 0 < \alpha, \beta \leq \frac12\}$ is uniformly integrable with respect to the measure $\PP_{\pi_2}$; i.e.
\begin{equation*}
\lim_{K \rightarrow \infty} \sup_{\alpha, \beta} \EE_{\pi_2} \left[ -\frac{N_{\alpha, \beta}}{\ln \alpha}\1_{\frac{N_{\alpha, \beta}}{-\ln \alpha} \geq - K} \right] = 0.
\end{equation*}
\end{lemma}
We use the following technical lemma which is Lemma 9 from \cite{kief16}.

\begin{lemma}\label{2016profilethm}
There is a number $c > 0$, computable in polynomial time, such that
\begin{equation*}
\PP_{\pi_2}\Big( L_{2|Q|n} \geq \exp( -\frac{c^2}{36} n ) \Big) \leq 4 \exp\Big( -\frac{c^2}{36} n \Big).
\end{equation*}
\end{lemma}

\begin{proof}[Proof of \Cref{unifintegofN}]
By \Cref{liexpmotivation}, conditioned on $E_\ell$ we have $\lim_{\alpha, \beta \rightarrow 0} \frac{N_{\alpha, \beta}}{\ln \alpha}$ exists $\PP_{\pi_2}$-almost surely. Hence, the convergence is also in $\PP_{\pi_2}$-measure. Therefore, by the Vitali convergence theorem \cite{bog2007} it is sufficient to show that the set of random variables $\{ \frac{N_{\alpha, \beta}}{\ln \alpha} \mid \alpha, \beta \in (0,\frac12) \}$ is uniformly integrable conditioned on $E_\ell$. In fact, because
\begin{equation}\label{unifintcond}
\lim_{K \rightarrow \infty} \sup_{\alpha, \beta} \EE_{\pi_2} \big[ \frac{N_{\alpha, \beta}}{- \ln \alpha}\1_{\frac{N_{\alpha, \beta}}{-\ln \alpha} \geq - K}\big] \geq \PP_{\pi_2}(E_\ell) \lim_{M \rightarrow \infty} \sup_{\alpha, \beta} \frac{1}{- \ln \alpha}\EE_{\pi_2} \big[ \frac{N_{\alpha, \beta}}{- \ln \alpha}\1_{\frac{N_{\alpha, \beta}}{-\ln \alpha} \geq - K} \mid E_\ell \big],
\end{equation}
It is sufficient to check the uniform integrability condition without conditioning on $E_\ell$.

For fixed $M \geq \frac{144|Q|}{c^2}$, write $m_\alpha = \floor{\frac{- M \ln \alpha}{2|Q|}}$. It follows that 
\begin{equation*}
\frac{2|Q|m_\alpha}{\ln \alpha} \leq M \text{ and } \alpha \geq \exp{-\frac{c^2}{36} m_\alpha}.
\end{equation*}
Further, $m_\alpha \geq \frac{M \ln 2}{2|Q|} - 1$. The following holds
\begin{align*}
& \quad \EE_{\pi_2} \big[ \frac{N_{\alpha, \beta}}{- \ln \alpha} \1_{\frac{N_{\alpha, \beta}}{-\ln \alpha} \geq M}\big]\\
& = \frac{1}{- \ln \alpha} \sum_{n = 0}^\infty \PP_{\pi_2} \big( N_{\alpha, \beta} \1_{N_{\alpha, \beta} \geq 2|Q|m_\alpha}> n\big) \\
& \leq \frac{2|Q|}{- \ln \alpha}\Big( m_\alpha ~\PP_{\pi_2}(N_{\alpha, \beta} \geq 2|Q|m_\alpha) + \sum_{n = m_\alpha}^\infty \PP_{\pi_2}\big( N_{\alpha, \beta} \geq 2|Q|n\big)\Big)\\
& \leq M \PP_{\pi_2}\big( L_{2|Q|m_{\alpha}} \geq \alpha\big) + \frac{2|Q|}{- \ln \alpha} \sum_{n = m_\alpha}^\infty \PP_{\pi_2}\big( L_{2|Q|n} \geq \alpha\big) \\
& \leq M \PP_{\pi_2}\big( L_{2|Q|m_{\alpha}} \geq \exp{-\frac{c^2}{36} m_\alpha}\big) + \frac{2|Q|}{- \ln \alpha} \sum_{n = m_\alpha}^\infty \PP_{\pi_2}\big( L_{2|Q|n} \geq \exp{-\frac{c^2}{36} n}\big) \\
& \leq 4M \exp{-\frac{c^2}{36}m_\alpha}  + \frac{8|Q|}{- \ln \alpha} \sum_{n = m_\alpha}^\infty \exp{-\frac{c^2}{36} n}\\
& \leq 4M \exp{-\frac{c^2}{36}m_\alpha}  + \frac{8|Q|\exp{-\frac{c^2}{36} m_\alpha}}{- \ln \alpha} \frac{1}{1 - \exp{c^2 / 36}}\\
& \leq 4M \exp{-\frac{c^2}{36} \Big(\frac{M \ln 2}{2|Q|} - 1 \Big)}  + \frac{8|Q|\exp \Big( -\frac{c^2}{36} (\frac{M \ln 2}{2|Q|} - 1 ) \Big) }{\ln 2} \frac{1}{1 - \exp{c^2 / 36}}\\
& \rightarrow 0
\end{align*}
as $M \rightarrow \infty$ where the fourth inequality follows by \Cref{2016profilethm}. Hence, \Cref{unifintcond} must hold.
\end{proof}

\liexpmotivation*

\begin{proof}
Since $\Psimin^n \leq L_n \leq \Psimin^{-n}$ it follows that
\begin{equation*}
N_{\alpha,\beta} \geq \frac{ \min \{ \ln \frac{\alpha}{1 - \beta}, \ln \frac{\beta}{1 - \alpha} \} }{\ln \Psimin}
\end{equation*}
Hence $N_{\alpha, \beta} \rightarrow \infty \ \PP_{\pi_2}$-almost surely as $\alpha, \beta \rightarrow 0$. Consider the case $\ell_k \in (-\infty, 0)$. Let $U_{\alpha,\beta} = \{w \in \Sigma^\omega \mid \ln L_{N_\alpha} \leq \ln \frac{ \alpha}{1 - \beta} \}$. The set $\bigcap_{\alpha, \beta \in (0,1]} U_{\alpha, \beta}^c \subseteq \{L_n \text{ is unbounded}\}$. Hence, $\lim_{\alpha, \beta \rightarrow 0}\1_{U_{\alpha,\beta}} = 1 \ ~\PP_{\pi_2}$-almost surely. Conditioned on $E_\ell$ it follows that

\begin{equation*}
0 \leq \1_{U_{\alpha, \beta}} \frac{\ln \frac{\alpha}{1 - \beta} - \ln L_{N_{\alpha, \beta}}}{N_\alpha} \leq \1_{U_{\alpha, \beta}}  \frac{\ln L_{N_{\alpha, \beta} - 1}  - \ln L_{N_{\alpha, \beta}}}{N_{\alpha, \beta}}\rightarrow 0 \ \text{ as } \alpha \rightarrow 0.
\end{equation*}
And so
\begin{equation*}
\lim_{\alpha, \beta \rightarrow 0}\frac{\ln \alpha}{N_{\alpha, \beta}} = \lim_{\alpha \rightarrow 0}\frac{\ln \frac{\alpha}{1-\beta}}{N_{\alpha, \beta}} =\lim_{\alpha \rightarrow 0}\frac{\ln L_{N_{\alpha, \beta}}}{N_{\alpha, \beta}} =\ell_k.
\end{equation*}
\end{proof}

\section{Proofs from \cref{sec:qual}} \label{app:qual}

\subsection{Proof of \cref{lem:expoprop}}

\lemexpoprop*
\begin{proof}
\begin{enumerate}
\item
Let $C \subseteq 2^Q \times Q$ be a bottom SCC of~$\B$.
Let $\pi, \pi'$ be distributions on~$Q$ and $q, q' \in Q$ such that $(\supp(\pi),q), (\supp(\pi'),q') \in C$.
Suppose that $\ell \in \Lambda_{\pi,e_{q}}$; i.e.,
\begin{equation} \label{eq:expoprop-ass}
 \PP_{e_{q}}\left(\lim_{n \to \infty} \frac1n \ln \frac{\| \pi \Psi(w_n) \|}{\| e_q \Psi(w_n) \|} = \ell\right) \ = \ x \quad \text{for some } x>0.
\end{equation}
It suffices to show that $\Lambda_{\pi',e_{q'}} = \{\ell\}$, i.e.,
\[
 \PP_{e_{q'}}\left(\lim_{n \to \infty} \frac1n \ln \frac{\| \pi' \Psi(w_n) \|}{\| e_{q'} \Psi(w_n) \|} = \ell\right) \ = \ 1.
\]
By L\'evy's 0-1 law it suffices to show that for all paths $q' a_1 q_1 \cdots a_m q_m$ with $\PP_{e_{q'}}(q' a_1 q_1 \cdots a_m q_m(\Sigma Q)^\omega) > 0$ there is $y>0$ with
\begin{equation} \label{eq:expoprop-goal}
 \PP_{e_{q'}}\left(\lim_{n \to \infty} \frac1n \ln \frac{\| \pi' \Psi(w_n) \|}{\| e_{q'} \Psi(w_n) \|} = \ell \;\middle\vert\; q' a_1 q_1 \cdots a_m q_m (\Sigma Q)^\omega \right) \ \ge \ y.
\end{equation}

Let $u = q' a_1 q_1 \cdots a_m q_m$ be a path with $\PP_{e_{q'}}(u (\Sigma Q)^\omega) > 0$.
Since $C$ is a bottom SCC of~$\B$, we have
\[
\PP_{e_{q'}}\left( \exists\,k \ge m : \supp(\pi' \Psi(a_1 \cdots a_m \cdots a_k)) = \supp(\pi),\ q_k = q \;\middle\vert\; u (\Sigma Q)^\omega) \right)=1\,.
\]
Thus, letting $v = q' a_1 q_1 \cdots a_m q_m \cdots a_k q_k$, with $k \ge m$, be an arbitrary extension of~$u$ with $\PP_{e_{q'}}(v (\Sigma Q)^\omega) > 0$ and $\supp(\pi' \Psi(a_1 \cdots a_m \cdots a_k)) = \supp(\pi)$ and $q_k = q$, we have
\begin{align}
 & \PP_{e_{q'}}\left(\lim_{n \to \infty} \frac1n \ln \frac{\| \pi' \Psi(w_n) \|}{\| e_{q'} \Psi(w_n) \|} = \ell \;\middle\vert\; u (\Sigma Q)^\omega \right) \notag\\
 \ge\ & \PP_{e_{q'}}\left(\lim_{n \to \infty} \frac1n \ln \frac{\| \pi' \Psi(w_n) \|}{\| e_{q'} \Psi(w_n) \|} = \ell \;\middle\vert\; v (\Sigma Q)^\omega \right) \notag\\
 \ge\ & \PP_{e_q}\left( \lim_{n \to \infty} \frac1n \ln \frac{\| (\pi' \Psi(a_1 \cdots a_k)) \Psi(w_n) \|}{ \| (e_{q'} \Psi(a_1 \cdots a_k)) \Psi(w_n) \|} = \ell \right) \notag\\
 \ge\ & \PP_{e_q}\left( \lim_{n \to \infty} \frac1n \ln \frac{\| (\pi' \Psi(a_1 \cdots a_k)) \Psi(w_n) \|}{ \| e_{q} \Psi(w_n) \|} = \ell \ \text{ and } \right. \label{eq:first-event} \\
      & \qquad\ \left. \lim_{n \to \infty} \frac1n \ln \frac{\| (e_{q'} \Psi(a_1 \cdots a_k)) \Psi(w_n) \|}{ \| e_{q} \Psi(w_n) \|} = 0 \right) \label{eq:second-event}
\end{align}
Concerning the event in~\eqref{eq:first-event}, by \eqref{eq:expoprop-ass} and since $\supp(\pi' \Psi(a_1 \cdots a_k)) = \supp(\pi)$, we have
\[
\PP_{e_q}\left( \lim_{n \to \infty} \frac1n \ln \frac{\| (\pi' \Psi(a_1 \cdots a_k)) \Psi(w_n) \|}{ \| e_{q} \Psi(w_n) \|} = \ell \right)
\ \ge\ x\,.
\]
Concerning the event in~\eqref{eq:second-event}, it follows from \cref{convergenceLn}.1 \stefan{better reference?} that
\[
\PP_{e_q}\left( \lim_{n \to \infty} \frac1n \ln \frac{\| (e_{q'} \Psi(a_1 \cdots a_k)) \Psi(w_n) \|}{ \| e_{q} \Psi(w_n) \|} \le 0 \right) \ = \ 1\,.
\]
Further, since $\PP_{e_{q'}}(q' a_1 q_1 \cdots a_k q_k (\Sigma Q)^\omega) > 0$, we have $q \in \supp(e_{q'} \Psi(a_1 \cdots a_k))$ and so
\[
\PP_{e_q}\left( \lim_{n \to \infty} \frac1n \ln \frac{\| (e_{q'} \Psi(a_1 \cdots a_k)) \Psi(w_n) \|}{ \| e_{q} \Psi(w_n) \|} \ge 0 \right) \ = \ 1\,.
\]
Thus, continuing the inequality chain from above, we conclude that
\[ \PP_{e_{q'}}\left(\lim_{n \to \infty} \frac1n \ln \frac{\| \pi' \Psi(w_n) \|}{\| e_{q'} \Psi(w_n) \|} = \ell \;\middle\vert\; u (\Sigma Q)^\omega \right) \ \ge\ x\,,
\]
proving~\eqref{eq:expoprop-goal}, as desired.
\item
Let $(S,q) \in C$ for a bottom SCC~$C$.
If $S = \emptyset$ then we may define $\ell(C) = -\infty$.
Otherwise, let $\pi_S$ denote the uniform distribution on~$S$.
Suppose that $\pi_S$ and~$e_q$ are not distinguishable.
By \cref{cor:probexp0} it follows that $0 \in \Lambda_{\pi_S,e_q}$.
Using part 1 we obtain $\ell(C) = 0$.
Finally, suppose that $\pi_S$ and~$e_q$ are distinguishable.
By \cref{cor:probexp0} it follows that $0 \not\in \Lambda_{\pi_S,e_q}$.
Since $C$ does not contain any states of the form $(\emptyset,q')$, by \cref{prop:neginf} we have $-\infty \not\in \Lambda_{\pi_S,e_q}$.
Using part~1 we obtain $\ell(C) \in (-\infty,0)$.
\item
We define a function~$f$ that maps paths of~$\H$ to paths of~$\B$ as follows.
Set $f(q_0 a_1 q_2 \cdots a_m q_m) := (S_0,q_0) (S_1,q_1) \cdots (S_m,q_m)$ where $S_0 = \supp(\pi_1)$ and $\delta(S_{i-1},a_i) = S_i$ for all $1 \le i \le m$.
The Markov chain~$\B$ is constructed so that for any path $v = (S_0,q_0) (S_1,q_1) \cdots (S_m,q_m)$ we have
\[
 \PP_\iota( v (2^Q \times Q)^\omega) \ = \ \PP_{\pi_2}(f^{-1}(v)(\Sigma Q)^\omega)\,.
\]
Let $C$ be any bottom SCC, and let $\ell = \ell(C)$.
Define the event
\[
 V_C := \{q_0 a_1 q_1 \cdots \in Q(\Sigma Q)^\omega \mid \exists\,m \in \NN : f(q_0 a_1 q_1 \cdots a_m q_m) \text{ ends in~$C$} \}\,.
\]
%Note that the $V_C$ are disjoint.
So we have $\PP_{\iota}(\{\text{visit $C$}\}) = \PP_{\pi_2}(V_C)$, and it suffices to show that $\PP_{\pi_2}(E_\ell \mid V_C) = 1$.
Let $u = q_0 a_1 q_1 \cdots a_m q_m$ be a path with $\PP_{\pi_2}(u (\Sigma Q)^\omega) > 0$ such that $f(u)$ ends in~$C$, say in $(S,q) \in C$, with $q = q_m$.
Thus, $\supp(\pi_1 \Psi(a_1 \cdots a_m)) = S$ and $q \in \supp(\pi_2 \Psi(a_1 \cdots a_m))$.
It suffices to show that $\PP_{\pi_2}(E_\ell \mid u (\Sigma Q)^\omega) = 1$.
We have:
\begin{align}
    & \PP_{\pi_2}(E_\ell \mid u (\Sigma Q)^\omega) \notag \\
 =\ & \PP_{\pi_2}\left(\lim_{n \to \infty} \frac1n \ln \frac{\| \pi_1 \Psi(w_n) \|}{\| \pi_2 \Psi(w_n) \|} = \ell \;\middle\vert\; u (\Sigma Q)^\omega \right) \notag \\
 =\ & \PP_{e_q}\left( \lim_{n \to \infty} \frac1n \ln \frac{\| (\pi_1 \Psi(a_1 \cdots a_m)) \Psi(w_n) \|}{ \| (\pi_2 \Psi(a_1 \cdots a_m)) \Psi(w_n) \|} = \ell \right) \notag \\
 \ge\ & \PP_{e_q}\left( \lim_{n \to \infty} \frac1n \ln \frac{\| (\pi_1 \Psi(a_1 \cdots a_m)) \Psi(w_n) \|}{ \| e_{q} \Psi(w_n) \|} = \ell \ \text{ and } \right. \label{eq:f-e}\\
      & \qquad\ \left. \lim_{n \to \infty} \frac1n \ln \frac{\| (\pi_2 \Psi(a_1 \cdots a_m)) \Psi(w_n) \|}{ \| e_{q} \Psi(w_n) \|} = 0 \right) \label{eq:s-e}
\end{align}
Concerning the event in~\eqref{eq:f-e}, by part~2 and since $\supp(\pi_1 \Psi(a_1 \cdots a_m)) = S$, we have
\[
\PP_{e_q}\left( \lim_{n \to \infty} \frac1n \ln \frac{\| (\pi_1 \Psi(a_1 \cdots a_m)) \Psi(w_n) \|}{ \| e_{q} \Psi(w_n) \|} = \ell \right) \ = \ 1\,.
\]
Concerning the event in~\eqref{eq:s-e}, it follows from \cref{convergenceLn}.1 \stefan{better reference?} that
\[
\PP_{e_q}\left( \lim_{n \to \infty} \frac1n \ln \frac{\| (\pi_2 \Psi(a_1 \cdots a_m)) \Psi(w_n) \|}{ \| e_{q} \Psi(w_n) \|} \le 0 \right) \ = \ 1\,.
\]
Further, since $q \in \supp(\pi_2 \Psi(a_1 \cdots a_m))$, we have
\[
\PP_{e_q}\left( \lim_{n \to \infty} \frac1n \ln \frac{\| (\pi_2 \Psi(a_1 \cdots a_m)) \Psi(w_n) \|}{ \| e_{q} \Psi(w_n) \|} \ge 0 \right) \ = \ 1\,.
\]
Thus, the events in \eqref{eq:f-e} and~\eqref{eq:s-e} occur $\PP_{e_q}$-a.s.
We conclude that $\PP_{\pi_2}(E_\ell \mid u (\Sigma Q)^\omega) = 1$, as desired. \qedhere
\end{enumerate}
\end{proof}

We can finally prove \Cref{asymptoticwald}. We use the fact that conditional expected time of visiting a state in a Markov chain is finite. This follows directly from the main result of \cite{shes13}.

\begin{proof}[Proof of \Cref{asymptoticwald}]
The first point follows by \Cref{unifintegofN} and \Cref{liexpmotivation} using Vitali's convergence theorem. The second point follows from \Cref{probexp0}. Finally, by \Cref{prop:neginf} we have $\sup_{\alpha, \beta}\EE_{\pi_2}[N_{\alpha, \beta} \mid E_{-\infty}] \leq \EE_{\pi_2}[N_\perp \mid E_{-\infty}] < \infty$ since by \Cref{lem:expoprop} $L_n = 0$ if and only if we visit a bottom SCC $C$ such that  $(\emptyset, q) \in C$ for some $q \in Q$. 
\end{proof}

\subsection{Proof of \cref{thm:qual-prob}} \label{app:thm-qual-prob}

Below we refer to the complexity class NC, the subclass of P comprising those problems solvable in polylogarithmic time by a parallel random-access machine using polynomially many processors; see, e.g., \cite[Chapter 15]{Pap94}.
To prove membership in PSPACE in a modular way, we use the following pattern:
\begin{lemma} \label{lem:PSPACE-transducer}
Let $P_1, P_2$ be two problems, where $P_2$ is in NC.
Suppose there is a reduction from $P_1$ to~$P_2$ implemented by a PSPACE transducer, i.e., a Turing machine whose work tape (but not necessarily its output tape) is PSPACE-bounded.
Then $P_1$ is in PSPACE.
\end{lemma}
\begin{proof}
Note that the output of the transducer is (at most) exponential.
Problems in NC can be decided in polylogarithmic space~\cite[Theorem~4]{Borodin77}.
Using standard techniques for composing space-bounded transducers (see, e.g., \cite[Proposition~8.2]{Pap94}), it follows that $P_1$ is in PSPACE.
\end{proof}

Now we prove the following theorem from the main body.
\thmqualprob*
\begin{proof}
\begin{enumerate}
\item
The Markov chain~$\B$ from \cref{lem:expoprop} is exponentially big but can be constructed by a PSPACE transducer, i.e., a Turing machine whose work tape (but not necessarily its output tape) is PSPACE-bounded.
The DAG (directed acyclic graph) structure, including the SCCs, of a graph can be computed in NL, which is included in NC.
Using the pattern of \cref{lem:PSPACE-transducer}, the DAG structure of the Markov chain~$\B$ can be computed in PSPACE.
Thus, there is a PSPACE transducer that computes both~$\B$ and its DAG structure.

For each bottom SCC~$C$, the PSPACE transducer also decides whether $\ell(C) = -\infty$ or $\ell(C) \in (-\infty,0)$ or $\ell(C) = 0$, using \cref{lem:expoprop}.2 and the polynomial-time algorithm for distinguishability from~\cite{kief14}.
Finally, to compute $\PP_{\pi_2}(E_{-\infty})$ and $\PP_{\pi_2}(E_0)$, by \cref{lem:expoprop}.3, it suffices to set up and solve a linear system of equations for computing hitting probabilities in a Markov chain.
This system can also be computed by a PSPACE transducer.
Linear systems of equations can be solved in NC~\cite[Theorem~5]{BorodinGathenHopcroft82}.
Using \cref{lem:PSPACE-transducer} again, we conclude that one can compute $\PP_{\pi_2}(E_{-\infty})$ and $\PP_{\pi_2}(E_0)$ in PSPACE.
\item This part was proved in the main body.
\item The claims concerning $\PP_{\pi_2}(E_{-\infty})$ follow from part~1 and \cref{prop:nontrivial-approx}.
Consider the problem whether $\PP_{\pi_2}(E_{0}) = 1$.
By part~1, it is in PSPACE.
Towards PSPACE-hardness we reduce again from mortality.
Let $(Q,\Sigma,\Phi)$ be an instance of the mortality problem.
Let $Q' := Q \cup \{q_\bot, q_2\}$ for fresh states $q_\bot,q_2$, and let $\Sigma' := \Sigma \cup \{\$\}$ for a fresh letter~$\$$.
Obtain $\Phi'$ from~$\Phi$ by adding, for every $q \in Q'$, a $\$$-labelled transition to~$q_\bot$, and an $a$-labelled loop from $q_2$ to itself for all $a \in \Sigma$.
Construct an HMM $(Q',\Sigma',\Psi)$ so that $\Phi'(a)$ and~$\Psi(a)$ have the same zero pattern for all $a \in \Sigma'$ (e.g., use uniform distributions).
See \cref{fig:PSPACE-hardness}.
\begin{figure}[ht]
\begin{center}
\begin{tikzpicture}[scale=2.5,LMC style]
\node[state] (q0) at (0,0) {$q_0$};
\node[state] (q1) at (1,0) {$q_1$};
\path[->] (q0) edge [loop,out=200,in=160,looseness=10] node[left] {$b$} (q0);
\path[->] (q0) edge node[above] {$a,b$} (q1);
\path[->] (q1) edge [loop,out=20,in=-20,looseness=10] node[right] {$a$} (q1);
\draw[->,line width=3] (1,-0.5) -- (1,-1);

\node[state] (q0') at (0,-1.5) {$q_0$};
\node[state] (q1') at (1,-1.5) {$q_1$};
\node[state] (q2) at (2,-1.5) {$q_2$};
\node[state] (qb) at (1,-2.5) {$q_\bot$};
\path[->] (q0') edge [loop,out=200,in=160,looseness=10] node[left] {$\frac14 b$} (q0');
\path[->] (q0') edge[bend left] node[above] {$\frac14 a$} (q1');
\path[->] (q0') edge[bend right] node[above] {$\frac14 b$} (q1');
\path[->] (q1') edge [loop,out=20,in=-20,looseness=10] node[right] {$\frac12 a$} (q1');
\path[->] (q0') edge node[left,xshift=-2] {$\frac14\$$} (qb);
\path[->] (q1') edge node[left] {$\frac12\$$} (qb);
\path[->] (q2) edge [loop,out=80,in=40,looseness=10] node[right] {$\frac13 a$} (q2);
\path[->] (q2) edge [loop,out=-40,in=-80,looseness=10] node[right] {$\frac13 b$} (q2);
\path[->] (q2) edge node[left] {$\frac13\$$} (qb);
\path[->] (qb) edge [loop,out=20,in=-20,looseness=10] node[right] {$1\$$} (qb);
\end{tikzpicture}
\end{center}
\caption{Illustration of the reduction from mortality to $\PP_{\pi_2}(E_0)<1$.
In this example, $\Phi(a b)$ is the zero matrix.
Accordingly, we have $\PP_{\pi_2}(E_0) < 1$, as $L_2(a b w) = 0$ for all $w \in \Sigma^\omega$.}
\label{fig:PSPACE-hardness}
\end{figure}
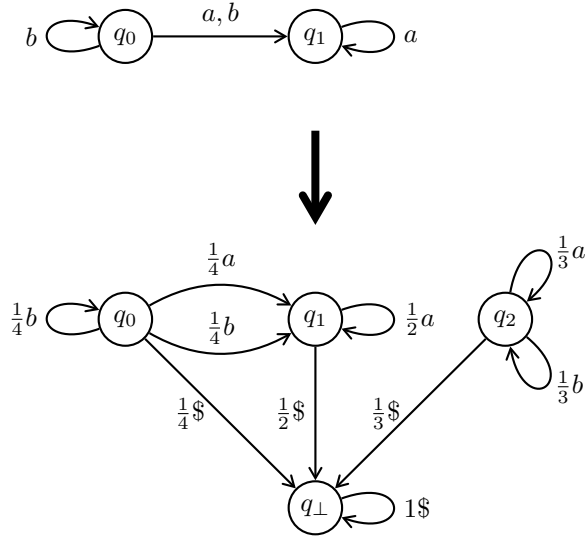
Let $\pi_1 \in [0,1]^{Q'}$ be the uniform distribution on~$Q$ (i.e., $(\pi_1)_{q_\bot} = (\pi_1)_{q_2} = 0$), and let $\pi_2$ be the Dirac distribution on~$q_2$.

Suppose $(Q,\Sigma,\Phi)$ is a positive instance of the mortality problem.
Let $v \in \Sigma^*$ such that $\Phi(v)$ is the zero matrix.
Then $L_{|v|}(v w) = 0$ holds for all $w \in \Sigma^\omega$.
It follows that $\PP_{\pi_2}(E_{-\infty})>0$ and so $\PP_{\pi_2}(E_0)<1$.

Conversely, suppose $(Q,\Sigma,\Phi)$ is a negative instance of the mortality problem.
The word produced from~$q_2$ contains $\PP_{e_{q_2}}$-a.s.\ the letter~$\$$, i.e., is of the form $u \$ v$ for $u \in \Sigma^*$ and $v \in (\Sigma \cup \{\$\})^\omega$.
Since $(Q,\Sigma,\Phi)$ is a negative instance, it follows that $\supp(\pi_1 \Psi(u \$)) = \{q_\bot\} = \supp(e_{q_\bot} \Psi(u \$))$.
Thus, $\lim_{n \to \infty} L_n > 0$. 
Hence, $\PP_{\pi_2}(E_0)=1$.
\qedhere
\end{enumerate}
\end{proof}

\section{Proofs from \Cref{sec:rep}}

For ease of reading, we repeat some definitions from \Cref{sec:rep}.

\generalizedlyapsysdef

\fromgentonongen*
The following is the key technical lemma that we use to prove \Cref{from-gen-to-nongen}. 

\begin{lemma}\label{makereduciblelyapsys}
	Let $\mathcal{H} = (Q, \Sigma, \Psi)$ be an HMM and define $\Psi'$ as in the main text such that $\mathcal{H}' = (Q, \Sigma Q, \Psi')$ is an HMM. One can compute in polynomial time a finite set $\Delta$, a probability distribution $\rho \in (0,1]^{\Delta}$ and for each $r_0 \in Q$ (a representation of) a mapping $\kappa_{r_0} : \Delta \rightarrow (\Sigma Q)$ with the following property. 
	
	Extend $\kappa_{r_0}$ to $(\Sigma Q)^+$ inductively. For each $c \in \Delta$ and $\overline{w} \in \Delta^+$ we let $\kappa_{r_0}(c~\overline{w}) = (a_1 r_1) \kappa_{r_1}(\overline{w})$ where $a_1 r_1 = \kappa_{r_0}(c)$. Then, for each word $w \in (\Sigma Q)^+$ we have
	\begin{equation}\label{markov2indepmeasures}
	\PP_{\rho}(\kappa_{r_0}^{-1}(\{w\}) \Delta^\omega) = \PP_{e_{r_0}}(w (\Sigma Q)^\omega)
	\end{equation}
	where $\PP_{e_{r_0}}$ refers to the measure on words produced by $\mathcal{H}'$ with initial distribution $e_{r_0}$.
\end{lemma}

\begin{proof}
	Let $\tau : [\Sigma Q] \rightarrow \Sigma Q$ be a bijection which we view as an arbitrary ordering on $\Sigma Q$ and write $\tau(k) = a_k q_k$. For each $r \in Q$ we define a function $\overline{\kappa}_r: [0,1) \rightarrow \Sigma Q$
	\begin{equation*}
	\overline{\kappa}_r(x) = a_K q_K \text{ when } \sum_{k = 1}^{K - 1} \Psi(a_k)_{r, q_k} \leq x < \sum_{k = 1}^K \Psi(a_k)_{r, q_k},
	\end{equation*}
	where for notational purposes, $\sum_{k = 1}^{0} \Psi(a_k)_{p, q_k} = 0$. For each $r \in Q$, since $\sum_{k = 1}^{|Q \times \Sigma|} \Psi(a_k)_{r, q_k} = 1$ it follows that $\overline{\kappa}_r$ is well defined on the interval $[0,1)$. Let $\Delta$ be the set of atomic elements of the finite $\sigma$-algebra $\sigma\{\overline{\kappa}_r^{-1}(\{a q\}) \mid r,q \in Q, a \in \Sigma\}$. The set $\Delta$ is a finite partition of $[0,1)$ and consists of intervals $[\alpha, \beta)$ where $\alpha, \beta \in [0,1]\cap\QQ$. We demonstrate the construction of $\Delta$ using the example HMM below.
	
	\begin{center}
		\begin{tikzpicture}[scale=2.5,LMC style]
		\node[state] (q1) at (-1,0) {$q_1$};
		\node[state] (q2) at (0,0) {$q_2$};
		\node[state] (q3) at (1,0) {$q_3$};
		
		\path[->] (q1) edge [loop,out=140,in=100,looseness=10] node[left] {$0.2 a$} (q1);
		\path[->] (q1) edge [loop,out=260,in=220,looseness=10] node[left] {$0.4 b$} (q1);
		\path[->] (q1) edge [bend left] node[above] {$0.4 b$} (q2);

		\path[->] (q2) edge [loop,out=110,in=70,looseness=10] node[above] {$0.4 a$} (q2);
		\path[->] (q2) edge [bend left] node[below] {$0.3 b$} (q1);
		\path[->] (q2) edge [bend left] node[above] {$0.3 b$} (q3);
		
		\path[->] (q3) edge [loop,out=80,in=40,looseness=10] node[right] {$0.6 a$} (q3);
		\path[->] (q3) edge [loop,out=-40,in=-80,looseness=10] node[right] {$0.2 b$} (q3);
		\path[->] (q3) edge [bend left] node[above] {$0.2 b$} (q2);
		\end{tikzpicture}
	\end{center}
	In the diagram below we give a representation of the functions $\overline{\kappa}_{q_1}, \overline{\kappa}_{q_2}, \overline{\kappa}_{q_3}$ and also the resulting set $\Delta$. The first three horizontal stacks of rectangles each represent a partition of the interval $[0,1)$ into the values taken by $\overline{\kappa}_{q_i}$ for $i = 1, 2, 3$. The bottom horizontal stack of rectangles represent the intervals in $\Delta$.
	\begin{center}
		\begin{tikzpicture}
		\draw[black, very thick] (0,-1.25) rectangle (2.2,-1.75);
		\draw[black, very thick] (2.2,-1.25) rectangle (4.4,-1.75);
		\draw[black, very thick] (4.4,-1.25) rectangle (6.6,-1.75);
		\draw[black, very thick] (6.6,-1.25) rectangle (7.7,-1.75);
		\draw[black, very thick] (7.7,-1.25) rectangle (8.8,-1.75);
		\draw[black, very thick] (8.8,-1.25) rectangle (11,-1.75);
		\node[] at (0,-1) {$0$};
		\node[] at (2.2,-1) {$0.2$};
		\node[] at (4.4,-1) {$0.4$};
		\node[] at (6.6,-1) {$0.6$};
		\node[] at (7.7,-1) {$0.7$};
		\node[] at (8.8,-1) {$0.8$};
		\node[] at (11,-1) {$1$};
		
		\node[] at (-0.5,-1.5) {$\Delta$};
		
		\draw [-latex](5.5,-0.25) -- (5.5,-1);
		
		\draw[black, very thick] (0,0) rectangle node{$a q_3$} (6.6,0.5);
		\draw[black, very thick] (6.6,0) rectangle node{$b q_3$} (8.8,0.5);
		\draw[black, very thick] (8.8,0) rectangle node{$b q_2$} (11,0.5);
		\node[] at (-0.5,0.25) {$\overline{\kappa}_{q_3}$};
		
		\draw[black, very thick] (0,0.75) rectangle node{$a q_2$} (4.4,1.25);
		\draw[black, very thick] (4.4,0.75) rectangle node{$b q_1$} (7.7,1.25);
		\draw[black, very thick] (7.7,0.75) rectangle node{$b q_3$} (11,1.25);
		\node[] at (-0.5,1) {$\overline{\kappa}_{q_2}$};
		
		\draw[black, very thick] (0,1.5) rectangle node{$a q_1$} (2.2,2);
		\draw[black, very thick] (2.2,1.5) rectangle node{$b q_1$} (6.6,2);
		\draw[black, very thick] (6.6,1.5) rectangle node{$b q_2$} (11,2);
		\node[] at (-0.5,1.75) {$\overline{\kappa}_{q_1}$};
		
		\draw (-0.2,-1.75) -- (-0.2,2);
		\end{tikzpicture}
		
	\end{center}
	
	One can compute in polynomial time the endpoints of all intervals $[\alpha, \beta) \in \Delta$. For any $a \in \Delta$ and $r \in Q$ the image $\overline{\kappa}_r(a)$ contains exactly one element. Therefore, for each $r \in Q$ we may define a function $\kappa_r: \Delta \rightarrow \Sigma Q$ such that $\kappa_r([\alpha, \beta)) = \overline{\kappa}_r(x)$ for all $x \in [\alpha, \beta)$. We define the probability distribution $\rho$ on $\Delta$ by $\rho_{[\alpha,\beta)} = \beta - \alpha$ for all $[\alpha,\beta) \in \Delta$. The distribution $\rho$ is computable in polynomial time.
	
	In our example above $\Delta = \{[0,0.2), [0.2,0.4), [0.4,0.6), [0.6,0.7), [0.7,0.8), [0.8,1)\}$ and $\kappa_{q_1}$ is defined piecewise by
	\begin{equation}
	\kappa_{q_1}(x) = \begin{cases}
	a q_1 & x = [0,0.2) \\
	b q_1 & x \in \{[0.2,0.4), [0.4,0.6)\} \\
	b q_2 & x \in \{[0.6,0.7), [0.7,0.8), [0.8, 1)\}. \\
	\end{cases} 
	\end{equation}

	Let $r_0 \in Q$. We may extend the mapping $\kappa_{r_0}$ to a word $c~\overline{w} \in \Delta^+$ inductively by letting $\kappa_{r_0}(c~\overline{w}) = (a_1 r_1) \kappa_{r_1}(\overline{w})$ where $(a_1 r_1) = \kappa_{r_0}(c)$. We now prove (\ref{markov2indepmeasures}) by induction on the length of the word. Let $a_1 r_1 \in \Sigma Q$ then by the definition of $\kappa_{r_0}$,
	\begin{equation*}
	\PP_{e_{r_0}}(a_1 r_1 (\Sigma Q)^\omega) = \Psi(a_1)_{r_0, r_1} = \rho\Big(\kappa_{r_0}^{-1}(\{a_1 r_1\})\Big) = \PP_{\rho}(\kappa_{r_0}^{-1}(\{a_1 r_1\}) \Delta^\omega).
	\end{equation*}
	Now assume \Cref{markov2indepmeasures} holds for all words of length $n - 1 \in \NN$. Then for a word $a_1 r_1 w \in (\Sigma Q)^n$
	\begin{align*}
	\PP_{e_{r_0}}(a_1 r_1 w (\Sigma Q)^\omega) & = \Psi(a_1)_{r_0, r_1} \PP_{e_{r_1}}(w (\Sigma Q)^\omega) \\
	& = \PP_{\rho}\Big(\kappa_{r_0}^{-1}(\{a_1 r_1\}) \Delta^\omega\Big) \PP_{\rho}\Big(\kappa_{r_1}^{-1}(\{w\}) \Delta^\omega\Big) \\
	& = \PP_{\rho}\Big(\kappa_{r_0}^{-1}(\{a_1 r_1\})\kappa_{r_1}^{-1}(\{w\})\Delta^\omega\Big) \\
	& = \PP_{\rho}\Big(\kappa_{r_0}^{-1}(\{a_1 r_1 w\}) \Delta^\omega\Big)
	\end{align*}
	by the independence of $\PP_{\rho}$ which completes the induction.
\end{proof}

For $\pi_1 \in [0,\infty)^{Q_1}$ and $\pi_2 \in [0,1]^{Q_2}$ we write $\pi_1 \times \pi_2 \in [0,\infty]^{Q_1 \times Q_2}$ for the vector $(\pi_1 \times \pi_2)_{(i,j)} := (\pi_1)_i  (\pi_2)_j$.

\begin{proof}[Proof of \Cref{from-gen-to-nongen}]
	By \Cref{makereduciblelyapsys}, since $(Q_2, \Sigma, \Psi_2)$ is an HMM, we may compute in polynomial time a finite set $\Delta$, a distribution $\rho \in [0,1]^{\Delta}$ and for each $r \in Q_2$ a mapping $\kappa_r : \Delta \rightarrow \Sigma Q_2$ with the property stated in \Cref{makereduciblelyapsys}. Then, we define $\tilde{\Psi}: \Delta \rightarrow [0,1]^{Q_1 \times Q_2}$ such that for all $(q_0, r_0), (q, r) \in Q_1 \times Q_2$:
	\begin{equation*}
	\tilde{\Psi}(\overline{a})_{(q_0, r_0), (q, r)} = \begin{cases}
	\Psi_1(a)_{q_0, q} & \text{when }\kappa_{r_0}(\overline{a}) = a r \\
	0 & \text{otherwise.}
	\end{cases}
	\end{equation*}
	We extend $\tilde{\Psi}$ to the mapping $\tilde{\Psi} : \Delta^* \rightarrow [0,1]^{(Q_1 \times Q_2) \times (Q_1 \times Q_2)}$ by $\tilde{\Psi}(\overline{a}_1 \cdots \overline{a}_n) = \tilde{\Psi}(\overline{a}_1) \cdots \tilde{\Psi}(\overline{a}_n)$. 
	
	Let $r_0 \in Q_2$, $\overline{a}_1 \cdots \overline{a}_n \in \Delta^*$ and $\kappa_{r_0}(\overline{a}_1 \cdots \overline{a}_n) = a_1 r_1 \cdots a_n r_n$. By the definition of the extension of $\kappa_{r_0}$, it follows that $\land_{i = 1}^n\kappa_{r_{i - 1}}(\overline{a}_i) = a_i r_i$. For all $q_0, \dots q_n \in Q_1$ we have
	\begin{align*}
	\Psi_1(a_1)_{q_0, q_1} \cdots \Psi_1(a_n)_{q_{n-1}, q_n} & = \tilde{\Psi}(\overline{a}_1)_{(q_0, r_0), (q_1, r_1)} \cdots\tilde{\Psi}(\overline{a}_n)_{(q_{n-1}, r_{n-1}), (q_n, r_n)} \\
	& = \sum_{\tilde{r}_1, \dots, \tilde{r}_n \in Q_2}\tilde{\Psi}(\overline{a}_1)_{(q_0, \tilde{r}_0), (q_1, \tilde{r}_1)} \cdots \tilde{\Psi}(\overline{a}_n)_{(q_{n-1}, \tilde{r}_{n-1}), (q_n, \tilde{r}_n)} \\
	\end{align*}
	and therefore by the definition of $\tilde{\Psi}$,
	\begin{equation}\label{equalityoftildeandpsi}
	\begin{aligned}
	& \|(\pi_1 \times e_{r_0}) \tilde{\Psi}(\overline{a}_1 \cdots \overline{a}_n)\| \\
	& = \sum_{q_0 \in \supp\,\pi_1}\sum_{(q_1, \tilde{r}_1), \dots (q_n, \tilde{r}_n) \in C} (\pi_1)_{q_0} \tilde{\Psi}(\overline{a}_1)_{(q_0, \tilde{r}_0), (q_1, \tilde{r}_1)} \cdots \tilde{\Psi}(\overline{a}_n)_{(q_{n-1}, \tilde{r}_{n-1}), (q_n, \tilde{r}_n)} \\
	& = \sum_{q_0 \in \supp\,\pi_1}\sum_{q_1, \dots q_n \in Q_1} \sum_{\tilde{r}_1, \dots \tilde{r}_n \in Q_2} (\pi_1)_{q_0} \tilde{\Psi}(\overline{a}_1)_{(q_0, \tilde{r}_0), (q_1, \tilde{r}_1)} \cdots \tilde{\Psi}(\overline{a}_n)_{(q_{n-1}, \tilde{r}_{n-1}), (q_n, \tilde{r}_n)} \\
	& = \sum_{q_0 \in \supp\,\pi_1}\sum_{q_1, \dots q_n \in Q_1} (\pi_1)_{q_0} \tilde{\Psi}(\overline{a}_1)_{(q_0, r_0), (q_1, r_1)} \cdots \tilde{\Psi}(\overline{a}_n)_{(q_{n-1}, r_{n-1}), (q_n, r_n)} \\
	& = \sum_{q_0 \in \supp\,\pi_1}\sum_{q_1, \dots q_n \in Q_1} (\pi_1)_{q_0} \Psi_1(a_1)_{q_0, q_1} \cdots \Psi_1(a_n)_{q_{n-1}, q_n} \\
	& = \|\pi_1 \Psi_1(a_1 \cdots a_n)\|.
	\end{aligned}
	\end{equation}
	Suppose there is an edge in $G_{(Q_1, \Sigma, \Psi_1), (Q_2, \Sigma, \Psi_2)}$ between two states $(q_0, r_0), (q, r) \in C$ then there is $a \in \Sigma$ such that $\Psi_1(a)_{q_0, q} > 0$ and $\Psi_2(a)_{r_0, r} > 0$. Hence, $\PP_{e_{r_0}}(a (\Sigma Q)^\omega) > 0$ and so by \Cref{makereduciblelyapsys} there exists $\overline{a} \in \kappa_{r_0}^{-1}(\{a r\})$. It follows that $\tilde{\Psi}(\overline{a})_{(q_0, r_0), (q, r)} > 0$. Therefore, since $C$ is a bottom SCC of $G_{(Q_1, \Sigma, \Psi_1), (Q_2, \Sigma, \Psi_2)}$, we have that the graph of the matrix system $(C, \Delta, \tilde{\Psi}_{|C})$ is strongly connected and therefore $((C, \Delta, \tilde{\Psi}_{|C}), \rho)$ is a Lyapunov system. Moreover, for any $\pi_1 \in [0,1]^{Q_1}$ and $\pi_2 \in [0,1]^{Q_2}$ such that $\supp(\pi_1) \times \supp(\pi_2) \subseteq C$ we have that $\| (\pi_1 \times \pi_2) \tilde{\Psi}(\overline{a}_1 \cdots \overline{a}_n) \| = \| (\pi_1 \times \pi_2)_{|C} \tilde{\Psi}_{|C}(\overline{a}_1 \cdots \overline{a}_n) \|$ where $(\pi_1 \times \pi_2)_{|C}$ is the truncation of the vector $\pi_1 \times \pi_2$ to $C$. Further,  we have for all $x \in \RR$ and $n \in \NN$
	\begin{equation}\label{equalityofdistributionspitorho}
	\begin{aligned}
	& \PP_{\pi_2}( \{ w \in \Sigma^\omega \mid \| \pi_1 \Psi(w_n) \| = x \}). \\
	& = \sum_{r_0 \in \supp\,\pi_2} (\pi_2)_{r_0} \PP_{e_{r_0}}( \{ w \in \Sigma^\omega \mid \| \pi_1 \Psi_1(w_n) \| = x \}) \\
	& = \sum_{r_0 \in \supp\,\pi_2} (\pi_2)_{r_0} \sum_{a_1 r_1 \cdots a_n r_n \in (\Sigma Q_2)^n \text{ s.t. }\| \pi_1 \Psi_1(a_1 \cdots a_n) \| = x} \PP_{e_{r_0}}(a_1 r_1 \cdots a_n r_n(\Sigma Q)^\omega) \\
	& = \sum_{r_0 \in \supp\,\pi_2} (\pi_2)_{r_0} \sum_{a_1 r_1 \cdots a_n r_n \in (\Sigma Q_2)^n \text{ s.t. }\| \pi_1 \Psi_1(a_1 \cdots a_n) \| = x} \PP_{\rho}(\kappa_{r_0}^{-1}(a_1 r_1 \cdots a_n r_n)(\Delta)^\omega) \\
	& \quad \text{by \Cref{makereduciblelyapsys}} \\
	& = \sum_{r_0 \in \supp\,\pi_2} (\pi_2)_{r_0} \PP_{\rho}(\{\overline{w} \in \Delta^\omega \mid \kappa_{r_0}(\overline{w}_n) = a_1 r_1 \cdots a_n r_n,~ \| \pi_1 \Psi_1(a_1 \cdots a_n) \| = x\}) \\
	& = \sum_{r_0 \in \supp\,\pi_2} (\pi_2)_{r_0} \PP_{\rho}( \{ \overline{w} \in \Delta^\omega \mid \| (\pi_1 \times e_{r_0}) \tilde{\Psi}(\overline{w}_n) \| = x \}) \quad \text{by (\ref{equalityoftildeandpsi})} \\
	& = \sum_{r_0 \in \supp\,\pi_2} (\pi_2)_{r_0} \PP_{\rho}( \{ \overline{w} \in \Delta^\omega \mid \| (\pi_1 \times e_{r_0})_{|C} \tilde{\Psi}_{|C}(\overline{w}_n) \| = x \}).
	\end{aligned}
	\end{equation}
	Combining the above equalities with \Cref{lem:lyapunov-exponent} we have that there exists a $\lambda \in [-\infty, 0]$ which does not depend on $\pi_1$ or $\pi_2$ such that 
	\begin{align*}
	1 & = \sum_{r_0 \in \supp \, \pi_2} (\pi_2)_{r_0} \\
	& = \sum_{{r_0} \in \supp \, \pi_2} (\pi_2)_{r_0} \PP_{\rho}\Big( \Big\{ \overline{w} \in \Delta^\omega \mid \lim_{n \rightarrow \infty} \frac1n \ln \| (\pi_1 \times e_{r_0})_{|C} \tilde{\Psi}_{|C}(\overline{w}_n) \| \in \{-\infty, \lambda\} \Big\}\Big) \text{ by \Cref{lem:lyapunov-exponent}} \\
	& = \PP_{\pi_2}\Big( \Big\{w \in \Sigma^\omega \mid \lim_{n \rightarrow \infty} \frac1n \ln \| \pi_1 \Psi(w_n) \| \in \{-\infty, \lambda\} \Big\}\Big) \text{ by (\ref{equalityofdistributionspitorho})}.
	\end{align*}
\end{proof}

\subsection{Proof of \Cref{lem:construct-L-systems}}

Towards a proof of \Cref{lem:construct-L-systems} we make some additional definitions. A \emph{reducible Lyapunov system} is a Lyapupov system $((Q, \Sigma, \Psi),\rho)$ except that the graph $G = (Q, E)$ with $E = \{(q, r) \mid \sum_{a \in \Sigma} \Psi_{q, r}(a) > 0\}$ is not necessarily strongly connected. For sets $U, V \subseteq Q$ we write $U \rightarrow_{G} V$ if there is $u \in U$ and $v \in V$ such that $v$ is reachable from $u$ in $G$.

We use the following technical lemmas. 

\begin{theorem}[F{\"u}rstenberg–Kesten theorem]\label{furstenbergkesten}
Let $((Q, \Sigma, \Psi), \mu)$ be a reducible Lyapunov system then there exists $\lambda \in [-\infty, \infty)$ such that $\PP_{\rho}$-a.s. we have
\begin{equation}
\lim_{n \rightarrow \infty} \frac1n \ln \|\vec{1} \Psi(w_n)\| = \lambda.
\end{equation}
\end{theorem}

The following lemma is Theorem 1.1 from \cite{germicorl08}.

\begin{lemma}\label{fullsupportlyapdecomp}
Let $\Sigma$ be a finite set of observations, $\rho \in [0,1]^\Sigma$ and let $((Q_1, \Sigma, \Psi_1), \rho), ((Q_2, \Sigma, \Psi_2), \rho)$ be reducible Lyapunov systems. By \Cref{furstenbergkesten} there exists $\lambda_1, \lambda_2 \in [-\infty, \infty)$ such that for $i = 1, 2$ we have
\begin{equation*}
\lim_{n\rightarrow \infty} \frac1n \ln \|\vec{1} \Psi_i(w_n) \| = \lambda_i.
\end{equation*}
Further, let $\Psi_{12} : \Sigma \rightarrow [0,1]^{Q_1 \times Q_2}$ specify transition from $Q_1$ to $Q_2$ and define $\Psi^* : \Sigma \rightarrow [0,1]^{(Q_1 \cup Q_2) \times (Q_1 \cup Q_2)}$ block-wise as 
\begin{equation*}
\Psi^*(a) = \begin{pmatrix}
\Psi_1(a) & \Psi_{12}(a) \\
0 & \Psi_2(a)
\end{pmatrix}.
\end{equation*}
It follows $((Q_1 \cup Q_2, \Sigma, \dot{\Psi}), \rho)$ is a reducible Lyapunov system and
\begin{equation*}
\lim_{n\rightarrow \infty} \frac1n \ln \|\vec{1} \dot{\Psi}(w_n) \| = \max \{\lambda_1, \lambda_2\}.
\end{equation*}
\end{lemma}

By generalising \Cref{fullsupportlyapdecomp}

\begin{lemma}\label{reduclyapsys}
Let $((Q, \Sigma, \Psi), \rho)$ be a reducible Lyapunov system and suppose $(Q, E)$ has strongly connected components $C_1, \dots, C_K$. Let $\,\mathbf{0}$ be the zero matrix and let
\begin{equation*}
\mathcal{Z} = \{C \in \{C_1, \dots , C_K\} \mid \exists u \in \Sigma^* \text{ s.t. }\Psi_{|C}(u) = \mathbf{0} \}.
\end{equation*}
For each $k \in [K]$, $\mathcal{G}_k := ((C_k, \Sigma, \Psi_{|C_k}), \rho)$ is a Lyapunov system. We have that $\PP_{\rho}$-almost surely,
\begin{equation} \label{furstenberggroups}
\lim_{n\rightarrow \infty} \frac1n \ln \|\vec{1} \Psi(w_n) \| = \max \Big( \Big\{\lambda(\mathcal{G}_k) \mid C_k \notin \mathcal{Z} \Big\} \cup \{-\infty\} \Big) 
\end{equation}
where $\vec{1}$ is the row vector all of whose entries are $1$. Further, for any initial distribution $\pi \in [0,1]^Q$, we have that $\PP_{\rho}$-almost surely,
\begin{equation}\label{generalizedfursten}
\lim_{n\rightarrow \infty} \frac1n \ln \|\pi \Psi(w_n) \| \in \Big\{\lambda(\mathcal{G}_k) \mid \supp~\pi  \rightarrow_{G} C_k, C_k \notin \mathcal{Z} \Big\} \cup \{-\infty\}.
\end{equation}
\end{lemma}

\begin{proof}
We first prove (\ref{furstenberggroups}). By \Cref{furstenbergkesten} for $i \in [N]$ there exists $\lambda_i \in [-\infty, \infty)$ such that $\PP_{\rho}$-a.s.,
\begin{equation*}
\lim_{n\rightarrow \infty} \frac1n \ln \|\vec{1} \Psi_{|C_i}(w_n) \| = \lambda_i.
\end{equation*}
In the case $\lambda_i = -\infty$ then by \Cref{lem:lyapunov-exponent} we have that $\PP_{\rho}$-a.s. there is a word $u \in \Sigma^*$ such that $\vec{1} \Psi_{|C_i}(u) = \vec{0}$ hence $C_i \in \mathcal{Z}$. 

In the case $\lambda_i \in \RR$ then by \Cref{lem:lyapunov-exponent} $\lambda_i = \lambda(\mathcal{G}_i)$. Also $\{w \in \Sigma^\omega \mid \exists n\in\NN : \vec{1}\Psi_{|C_i}(w_n) = \vec{0}\}$ is $\PP_{\rho}$-null set and therefore empty because $\rho$ has full support and therefore all finite words have positive probability of being produced. It follows that $C_i \notin \mathcal{Z}$.

We proceed by induction. In the case $K = 1$, the cases described previously immediately imply that (\ref{furstenberggroups}) holds. Now we assume the lemma holds for some $K = m \in \NN$. We prove this implies the theorem holds for $K = m + 1$.

Let $((Q, \Sigma, \Psi), \rho)$ be a reducible Lyapunov system such that $(Q, E)$ has strongly connected components $C_1, \dots, C_{m + 1}$. We have that $((C_1 \cup \cdots \cup C_m, \Sigma, \Psi_{|C_1 \cup \cdots \cup C_m}), \rho)$ is a reducible Lyapunov system so by the induction hypothesis, $\PP_{\rho}$-a.s. we have
\begin{equation*}
\lim_{n\rightarrow \infty} \frac1n \ln \|\vec{1} \Psi_{|C_1 \cup \cdots \cup C_m}(w_n) \| = \max\Big(\{\lambda(\mathcal{G}_i) \mid i \in [m], C_i \notin \mathcal{Z}\} \cup \{-\infty\}\Big).
\end{equation*}
In the case that $\lambda_{m + 1} \in \RR$ then $\lambda(\mathcal{G}_{m+1}) = \lambda_{m+1}$ and $C_{m + 1} \notin \mathcal{Z}$. Therefore, $\lambda_{m + 1} \in \{\lambda(\mathcal{G}_i) \mid i \in [m + 1], C_i \notin \mathcal{Z}\}$.

In the case $\lambda_{m + 1} = -\infty$, then $C_{m + 1} \in \mathcal{Z}$ and clearly $\{\lambda(\mathcal{G}_i) \mid i \in [m], C_i \notin \mathcal{Z}\} = \{\lambda(\mathcal{G}_i) \mid i \in [m + 1], C_i \notin \mathcal{Z}\}$. So by \Cref{fullsupportlyapdecomp} we have
\begin{align*}
\lim_{n\rightarrow \infty} \frac1n \ln \|\vec{1} \Psi(w_n) \| & = \max\Big\{\lambda_{m + 1}, \max\Big(\{\lambda(\mathcal{G}_i) \mid i \in [m], C_i \notin \mathcal{Z}\} \cup \{-\infty\} \Big) \Big\} \\
& = \max \Big(\{\lambda(\mathcal{G}_i) \mid i \in [m + 1], C_i \notin \mathcal{Z}\} \cup \{-\infty\} \Big).
\end{align*}
We now prove (\ref{generalizedfursten}). Like in \Cref{sec:qual}, for $S \subseteq Q$ and $a\in\Sigma$, define $\delta'(S, a) = \{q' \in Q \mid \exists\, q \in S : \Psi(a)_{q, q'} > 0\}$. Then we define the Markov chain $\mathcal{B}' := (2^Q, T')$ where
\begin{equation*}
T'_{S, S'} := \sum_{\delta'(S, a) = S'} \rho(a).
\end{equation*}
Let $\mathcal{D}$ be the set of bottom SCCs that are reachable in $\mathcal{B}'$ from the initial state $\supp~\pi$. We fix $D \in \mathcal{D}$. We may order the $C_1 \cdots C_K$ such that for some $k \in [K]$
\begin{equation*}
S_D := \{q \in S \mid S \in D \} \cap C_i \neq \emptyset
\end{equation*}
for all $i \in [k]$. It is clear that $S_D \subseteq C_1\cup \cdots \cup C_k$. We show the reverse inclusion. Let $i \in [k]$ then there is $q_0 \in C_i \cap S_D$. For any $q' \in C_i$ there is a word $a_1 q_1 \cdots a_m q_m \in (\Sigma Q)^*$ such that $q_m = q'$ and $\Psi_{q_{i - 1}, q_i}(a_i) > 0$ for all $i \in [m]$. It follows that $q' \in S_D$ and therefore $S_D = C_1 \cup \cdots \cup C_k$.

We have that $\PP_{\rho}$-a.s. $\lim_{n \rightarrow \infty} \frac1n \ln \|\vec{1} \Psi_{|S_D}(w_n)\|$ exists and equals 
\begin{equation*}
\lambda_D := \max\Big(\{\lambda(\mathcal{G}_i) \mid i \in [k], C_i \notin \mathcal{Z}\} \cup \{-\infty\}\Big)
\end{equation*}
by (\ref{furstenberggroups}). In particular there is $q \in S_D$ such that
\begin{equation*}
\alpha := \PP_{\rho}\Big(\lim_{n \rightarrow \infty} \frac1n \ln \|e_q \Psi(w_n)\| = \lambda_D\Big) = \PP_{\rho}\Big(\lim_{n \rightarrow \infty} \frac1n \ln \|e_q \Psi_{|S_D}(w_n)\| = \lambda_D\Big) > 0.
\end{equation*} 
For all $S_0 \in D$ there is a sequence $a_1 S_1 \cdots a_n S_n \in (\Sigma D)^*$ such that $q \in S_n$ and $S_{i} = \delta(S_{i - 1}, a_i)$ for all $i = 1, \dots, n$. We have that $\beta_{S_0} := \prod_{i = 1}^n \rho(a_i) > 0$. Write $\beta = \min \{\beta_{S_0} \mid S_0 \in D\}$. 
Let $\pi' \in [0,1]^Q$ be such that $\supp~\pi' \in D$. Then for any $u \in \Sigma^*$ we have 
\begin{align*}
& \PP_{\rho}(\lim_{n \rightarrow \infty} \frac1n \ln \|\pi' \Psi(w_n)\| = \lambda_D \mid u \Sigma^\omega) \\
& = \PP_{\rho}(\lim_{n \rightarrow \infty} \frac1n \ln \|\pi' \Psi(w_n)\| = \lambda_D \mid u \Sigma^\omega) \\
& \geq \beta \PP_{\rho}(\lim_{n \rightarrow \infty} \frac1n \ln \|\pi' \Psi(u w_n)\| = \lambda_D) \\
& \geq \beta \PP_{\rho}(\lim_{n \rightarrow \infty} \frac1n \ln \|e_q \Psi(w_n)\| = \lambda_D) \\
& = \beta \alpha > 0.
\end{align*}
Hence by L{\'e}vy's 0-1 law we have 
\begin{align*}
1 & = \PP_{\rho}\Big(\lim_{n \rightarrow \infty} \frac1n \ln \|\pi' \Psi(w_n)\| = \lambda_D\Big) \\
& = \PP_{\rho}\Big(\lim_{n \rightarrow \infty} \frac1n \ln \|\pi' \Psi(w_n)\| \in \{\lambda(\mathcal{G}_i) \mid \supp~\pi \rightarrow_G C_i, C_i \notin \mathcal{Z}\} \cup \{-\infty\}\Big)
\end{align*}
where the last equality follows because $D \in \mathcal{D}$ is reachable from $\supp~\pi$ in $\mathcal{B}'$ which implies that for any $C_i \subseteq S_D$ we have $\supp~\pi \rightarrow_G C_i$.

We now relax the dependence on $D$. For any $D \in \mathcal{D}$, let $F_D = \{w \in \Sigma^\omega \mid \exists n\in\NN : \supp~(\pi \Psi(w_n)) \in D\}$. Then,
\begin{equation*}
\PP_{\rho}\Big( \bigcup_{D \in \mathcal{D}}F_D\Big) = 1.
\end{equation*}
Write $\Theta = \{\lambda(\mathcal{G}_i) \mid \supp~\pi \rightarrow_G C_i, C_i \notin \mathcal{Z}\} \cup \{-\infty\}$ then $\PP_{\rho}$-a.s. we have
\begin{align*}
\PP_{\rho}(\lim_{n \rightarrow \infty} \frac1n \ln \|\pi \Psi(w_n)\| \in \Theta) & = \sum_{D \in \mathcal{D}} \PP_{\rho}(\{\lim_{n \rightarrow \infty} \frac1n \ln \|\pi \Psi(w_n)\| \in \Theta \} \mid F_D) \PP_{\rho}(F_D) \\
& = \sum_{D \in \mathcal{D}} \PP_{\rho}(F_D) = 1.
\end{align*}
\end{proof}

The following is also repeated from \Cref{sec:rep}.

\definitionsforgenlyap

\constructLsystems*

\begin{proof}
Let $\mathcal{H} = (Q, \Sigma, \Psi)$ and observe that 
\begin{equation*}
\PP_{\pi_2}\Big(r_0 a_1 r_1 a_2 r_2 \cdots \in Q (\Sigma Q)^\omega \mid \exists k \in \NN \text{ s.t. } r_k \text{ is in a bottom SCC of } \mathcal{H} \Big) = 1.
\end{equation*}
Consider a word $a_1 r_1 a_2 r_2 \cdots \in (\Sigma Q)^\omega$ such that $\PP_{\pi_2}(Q a_1 r_1 \cdots a_m r_m (\Sigma Q)^\omega) > 0$ for all $m \in \NN$ and further, there exists a $k \in \NN$ such that $r_k \in P$ where $P \subseteq Q$ is a bottom SCC of $\mathcal{H}$. We write $u = a_1 r_1 \cdots a_k r_k$ and $w = a_{k+1} r_{k+1} a_{k+2} r_{k+2}\cdots \in (\Sigma Q)^\omega$. Let $\mu_i = \pi_i \Psi(u) \times e_{r_k}$ for $i = 1,2$. Recall the definition of $\overline{\Psi}$. We have that for any $n > k$ that $\| \pi_i \Psi(a_1 \cdots a_{k + n}) \| = \| \pi_i \Psi(a_1 \cdots a_k) \Psi(a_{k+1} \cdots a_{k + n}) \| = \| \mu_i \overline{\Psi}(w_n) \|$ for $i = 1, 2$. We fix $u$ and consider words $w$ produced by $\mathcal{H}$ with initial distribution $e_{r_k}$. We have that $\PP_{e_{r_k}}$-almost surely if the limits exist,
\begin{align*}
\lim_{n \rightarrow \infty} \frac{1}{k + n} \ln L_{k + n}(uw) & = \lim_{n \rightarrow \infty}  \frac{1}{k + n} \Big[ \ln \| \pi_1 \Psi(a_1 \cdots a_{k + n}) \| - \ln \|\pi_2 \Psi(a_1 \cdots a_{k + n}) \| \Big] \\
& = \lim_{n \rightarrow \infty} \frac{1}{k + n} \Big[ \ln \| \mu_1 \overline{\Psi}(w_n) \| - \ln \|\mu_2 \overline{\Psi}(w_n) \| \Big] \\
& = \lim_{n \rightarrow \infty} \frac1n \Big[ \ln \| \mu_1 \overline{\Psi}(w_n) \| - \ln \|\mu_2 \overline{\Psi}(w_n) \| \Big] \\
\end{align*}
since $\lim_{n \rightarrow \infty} \frac{n}{k + n} = 1$.

By \Cref{makereduciblelyapsys} there is a finite set $\Delta$, distribution $\rho \in (0,1]^{\Delta}$ and mapping $\kappa_{r_k}: \Delta^+ \rightarrow (\Sigma Q)^+$ such that $\PP_{\rho}(\{\kappa_{r_k}^{-1}(w_n \}) \Delta^\omega) = \PP_{e_{r_k}}(w_n (\Sigma Q)^\omega)$ for all $n \in \NN$. Let $S \subseteq Q \times Q$. It follows that for each $n \in \NN$ and $A \in [0, 1]^{S\times S}$ we have
\begin{equation}\label{equalityofpushforward}
\PP_{e_{r_k}}(\{w \in (\Sigma Q)^\omega \mid \overline{\Psi}_{|S}(w_n) = A\}) = \PP_{\rho}(\{\overline{w} \in \Delta^\omega \mid \overline{\Psi}_{|S} (\kappa_{r_k}(\overline{w}_n)) = A \}).
\end{equation}

Let $\mathcal{C}$ be the set of SCCs of $G_{\mathcal{H}, \mathcal{H}}$. We define $\sigma : \mathcal{C} \rightarrow \mathcal{P}(Q)$ such that $\sigma(R) = \{q \in Q \mid \exists p \in Q \text{ s.t. }(p,q) \in R\}$. Recall that $r_k \in P$ and $P$ is a bottom SCC. Thus, for all $i = 1, 2$, $n \in \NN$ and $(p, q) \in \supp~ ( \mu_i \overline{\Psi}(w_n))$ we have $q \in P$. Therefore, for all $R \in \mathcal{C}$ 
\begin{equation}\label{staysinP}
\supp~\mu_i \rightarrow_{G_{\mathcal{H}, \mathcal{H}}} R \implies \sigma(R) = P.
\end{equation}
Consider the composition $\overline{\Psi} \circ \kappa_{r_k} : \Delta^+ \rightarrow [0,1]^{(Q \times Q) \times (Q \times Q)}$. For any $R \in \sigma^{-1}(P)$ we have that trivially 
\begin{equation}\label{interchangeablerestandcomp}
(\overline{\Psi} \circ \kappa_{r_k})_{|R} = \overline{\Psi}_{|R} \circ \kappa_{r_k}.
\end{equation}
Further, $((R, \Delta, \overline{\Psi}_{|R} \circ \kappa_{r_k}), \rho)$ is a Lyapunov system and recall that $S^1_R = ((R, \Sigma P, \overline{\Psi}_{|R}), (P, \Sigma P, \Psi'_{|P}), C^1_R)$ is a generalized Lyapunov system where $C^1_R = \{((q_1, q_2), q_2) \mid (q_1, q_2) \in R\}$. Let $V = \{q \in Q \mid (q, r_k) \in R\}$ and $\upsilon \in [0,1]^V$. Since $\supp(\upsilon \times e_{r_k}) \times \{r_k\} \in C^1_R$ by \Cref{from-gen-to-nongen} the limit $\lim_{n \rightarrow \infty} \frac1n \ln \|(\upsilon \times e_{r_k}) \overline{\Psi}_{|R}(w_n)\|$ exists $\PP_{e_{r_k}}$-almost surely and equals either $\lambda(S_R^1)$ or $-\infty$. Then, by (\ref{equalityofpushforward}) with $S = R$ we have for all $x \in [-\infty, 0]$ that 
\begin{align}\label{equivofgeneralizedandnorm}
& \PP_{e_{r_k}}(\{w \in (\Sigma Q)^\omega \mid \lim_{n \rightarrow \infty} \frac1n \ln \|(\upsilon \times e_{r_k}) \overline{\Psi}_{|R}(w_n)\| = x\}) \nonumber\\
& = \PP_{\rho}(\{\overline{w} \in \Delta^\omega \mid \lim_{n \rightarrow \infty} \frac1n \ln \|(\upsilon \times e_{r_k}) \big[(\overline{\Psi}_{|R} \circ \kappa_{r_k})(\overline{w}_n)\big]\| = x\}).
\end{align}
Recall the definition of $\mathcal{L}$. The following series of equalities hold:
\begin{equation}\label{possiblelyaplimits}
\begin{aligned}
& \lim_{n \rightarrow \infty} \frac1n \ln \|\mu_i \overline{\Psi}(w_n) \| \quad \PP_{e_{r_k}}\text{-a.s.}\\
& = \lim_{n \rightarrow \infty} \frac1n \ln \|\mu_i \big[(\overline{\Psi} \circ \kappa_{r_k})(\overline{w}_n)\big] \| \quad \PP_{\rho}\text{-a.s.}  \quad \text{ by (\ref{equalityofpushforward}) where } S = Q \times Q\\ 
& \in \Big\{\lambda((R, \Delta, (\overline{\Psi} \circ \kappa_{r_k})_{|R}), \rho) \mid R \in \mathcal{C} / \mathcal{L},  \supp~\mu_i \rightarrow_{G_{\mathcal{H}, \mathcal{H}}} R \Big\} \cup \{-\infty\} \, \PP_{\rho}\text{-a.s.} \\
& \quad \quad \text{by \Cref{reduclyapsys}}\\
& = \Big\{\lambda((R, \Delta, \overline{\Psi}_{|R} \circ \kappa_{r_k}), \rho) \mid R \in \sigma^{-1}(P) / \mathcal{L},  \supp~\mu_i \rightarrow_{G_{\mathcal{H}, \mathcal{H}}} R \Big\} \cup \{-\infty\} \\
& \quad \quad \text{by (\ref{staysinP}) and (\ref{interchangeablerestandcomp})}\\ 
& = \Big\{\lambda(\S_R^1) \mid R \in \sigma^{-1}(P) / \mathcal{L},  \supp~\mu_i \rightarrow_{G_{\mathcal{H}, \mathcal{H}}} R \Big\} \cup \{-\infty\} \quad \text{by (\ref{equivofgeneralizedandnorm})}.\\
& = \Big\{\lambda(\S_R^1) \mid R \in \sigma^{-1}(P) /\mathcal{L},  \supp~\mu_i \rightarrow_{G_{\mathcal{H}, \mathcal{H}}} R \Big\} \cup \{-\infty\}.\\
\end{aligned}
\end{equation}

We now focus on the case $i = 2$. Since $P$ is an SCC, there is a unique right-bottom SCC $P'$ that contains the states $\{(p, p) \mid p \in P\}$. Let $U'$ be another right-bottom SCC in the set $\sigma^{-1}(P) /\mathcal{L}$. Since $U'$ is an SCC and not in $\mathcal{L}$ we have that $\|\sum_{(s, r) \in U'}(e_s \times e_r)\overline{\Psi}_{|U'}(w_n)\| > 0$ for all $n \in \NN$. Therefore $\PP_{e_r}$-a.s. by \Cref{from-gen-to-nongen}, 
\begin{align*}
\lambda(\S_{U'}^1) & = \lim_{n \rightarrow \infty} \frac1n \ln \|\sum_{(s, r) \in U'}(e_s \times e_r)\overline{\Psi}_{|U'}(w_n)\| \\
& \leq \lim_{n \rightarrow \infty} \frac1n \ln |U'|\max_{(s, r) \in U'} \|(e_s \times e_r)\overline{\Psi}_{|U'}(w_n)\| \\
& = \max_{(s, r) \in U'} \lim_{n \rightarrow \infty} \frac1n \ln \|(e_s \times e_r)\overline{\Psi}_{|U'}(w_n)\| \\
& = \lambda(\S_{U'}^1) \text{ because each limit is either } \lambda(\S^1_{U'})\text{ or }-\infty 
\end{align*}
It follows that there is some $(s, r) \in U'$ such that $\PP_{e_r}(\lim_{n \rightarrow \infty} \frac1n \ln \|(e_s \times e_r)\overline{\Psi}_{|U'}(w_n)\| = \lambda(\S_{U'}^1)) > 0$.
Let $U = \{q \in Q \mid \exists p\in Q \text{ s.t. }(q,p) \in U'\}$. Then $s \in U$, $r \in P$ and $(r, r) \in P'$. Hence, with $\PP_{e_r}$-probability greater than $0$ we have
\begin{align*}
0 & \geq \lim_{n \rightarrow \infty} \frac1n \ln \frac{\|e_s \Psi_{|U}(w_n)\|}{\|e_r \Psi_{|P}(w_n)\|} \text{ by \Cref{convergenceLn}}\\
& = \lim_{n \rightarrow \infty} \frac1n \ln \frac{\|(e_s \times e_r) \overline{\Psi}_{|U'}(w_n)\|}{\|(e_r \times e_r) \overline{\Psi}_{|P'}(w_n)\|} \\
& = \lim_{n \rightarrow \infty} \frac1n \big[ \ln \|(e_s \times e_r) \overline{\Psi}_{|U'}(w_n)\| - \ln \|(e_r \times e_r) \overline{\Psi}_{|P'}(w_n)\| \big] \\
& = \lambda(\S_{U'}^1) - \lambda(\S_{P'}^1)
\end{align*}
which implies that $\lambda(\S^1_{P'}) \geq \lambda(\S^1_{U'})$. We have that $(r_k, r_k) \in \supp~ \mu_2$ and so $\PP_{e_{r_k}}$-almost surely we have ${\| (e_{r_k} \times e_{r_k}) \overline{\Psi}_{P'}(w_n) \| > 0}$ for all $n \in \NN$. By (\ref{staysinP}), it follows that $\PP_{e_{r_k}}$-a.s. we have
\begin{align*}
\lambda(\S_{P'}^1) & = \lim_{n \rightarrow \infty} \frac1n \ln \|(e_{r_k} \times e_{r_k}) \overline{\Psi}_{|P'}(w_n) \| \quad \text{ by \Cref{from-gen-to-nongen}}\\
& \leq \lim_{n \rightarrow \infty} \frac1n \ln \|\mu_2 \overline{\Psi}(w_n) \|\\
& \leq \max \Big\{\lambda(\S_R^1) \mid R \in \sigma^{-1}(P),~\supp~\mu_2 \rightarrow_{G_{\mathcal{H}, \mathcal{H}}} R \Big\} \text{ by (\ref{possiblelyaplimits})}\\
& \leq \lambda(\S_{P'}^1).
\end{align*}
Then, recalling the definition of $S_R^2$ we have $\PP_{\pi_2}$-almost surely
\begin{align*}\label{fixedutheoremversion}
\lim_{n \rightarrow \infty} \frac{1}{k + n} \ln L_{k + n}(uw) & = \lim_{n \rightarrow \infty} \frac1n \Big[ \ln \| \mu_1 \overline{\Psi}(w_n) \| - \ln \|\mu_2 \overline{\Psi}(w_n) \| \Big] \nonumber\\
& \in \Big\{\lambda(\S_R^1) - \lambda(S_{P'}^1) \mid R \in \sigma^{-1}(P), \supp~ \mu_1 \rightarrow_{G_{\mathcal{H}, \mathcal{H}}} R \Big\}\cup \{-\infty\} \\
& = \Big\{\lambda(\S_R^1) - \lambda(S_R^2) \mid R \in \sigma^{-1}(P), \supp~ \mu_1 \rightarrow_{G_{\mathcal{H}, \mathcal{H}}} R \Big\}\cup \{-\infty\}
\end{align*}
since $\S_R^2 = \S_{P'}^1$ when $\sigma(R) = P$.

Recall that for any $R \in \mathcal{R}$ the right-bottom SCC $R'$ is such that $\{(p, p) \mid p \in \sigma(R)\} \subseteq R'$

We may now relax the dependence on $u$ and $P$. We have that $\supp~\pi_1 \times \supp~\pi_2 \rightarrow_{G_{\mathcal{H}, \mathcal{H}}} (\supp~ \pi_1 \Psi(a_1 \cdots a_k)) \times \{r_k\}$ for any $a_1 r_1 \cdots a_k r_k \in (\Sigma Q)^+$ where $\PP_{\pi_2}(Q a_1 r_1 \cdots a_k r_k (\Sigma Q)^\omega) > 0$. Therefore for all $R \in \mathcal{R}$ 
\begin{equation*}
(\supp~\pi_1 \Psi(a_1 \cdots a_k)) \times \{r_k\} \rightarrow_{G_{\mathcal{H}, \mathcal{H}}} R \implies  \supp~\pi_1 \times \supp~\pi_2 \rightarrow_{G_{\mathcal{H}, \mathcal{H}}} R.
\end{equation*}
Finally we have up to a $\PP_{\pi_2}$-null set
\begin{align*}
& \{a_1 r_1 a_2 r_2 \cdots \in (\Sigma Q)^\omega \mid \exists k \in \NN \text{ s.t. } r_k \text{ is in a bottom SCC of } \mathcal{H}\} \\
& \subseteq \Big\{\lim_{n \rightarrow \infty} \frac{1}{k + n} \ln L_{k + n}(uw) \\
& \in \{-\infty\} \cup \big\{\lambda(\S_R^1) - \lambda(\S_R^2) \mid R \in \mathcal{R}, \supp~\pi_1 \times \supp~\pi_2 \rightarrow_{G_{\mathcal{H}, \mathcal{H}}} R \big\} \Big\}
\end{align*}
and the lemma follows.
\end{proof}

\section{Proofs from \cref{sec:det}} \label{app:det}

\thmdet*
\begin{proof}
In a Markov chain, one can compute the stationary distribution and hitting probabilities in polynomial time by solving a linear system of equations.
Thus, the numbers $\ell(C)$ defined before \cref{lem:polyprop} can be computed in polynomial time.
Both parts of the theorem follow then from \cref{lem:polyprop}.
A slight complication is that for part~2, for an $\ell = \sum_i x_i \ln y_i \in \Lambda_{\pi_1,\pi_2}$, in order to compute $\PP_{\pi_2}(E_\ell)$ we have to sum the hitting probabilities for all $C$ with $\ell = \ell(C)$.
To select those~$C$ we have to compare numbers of the form $\sum_i x_i \ln y_i$ where $x_i,y_i \in \QQ$, and it is not immediately obvious how to do that.
However, one can compare two such numbers for equality in polynomial time as shown in~\cite{EtessamiSY14}.
\end{proof}

\end{document}